\documentclass{amsart}
\usepackage{amsmath,amsthm}
\usepackage{amsfonts,amssymb}

\usepackage{enumerate}

\newtheorem{thm}{Theorem}[section]
\newtheorem{cor}[thm]{Corollary}
\newtheorem{lem}[thm]{Lemma}
\newtheorem{prop}[thm]{Proposition}

\newtheorem{defn}[thm]{Definition}

\theoremstyle{remark}

\def\sph{\mathbb{S}^{d-1}}
\def\ball{\mathbb{B}^{d}}
\def\f{\frac}

 \def\a{{\alpha}} 
 \def\b{{\beta}}

 \def\t{{\theta}}
 \def\l{{\lambda}}
 
 \def\o{{\omega}}
 \def\s{{\sigma}}
 \def\la{{\langle}}
 \def\ra{{\rangle}}

 \def\d{{\textnormal{d}}}

 \def\CH{{\mathcal H}}

 \def\P{{\mathcal P}}

 \def\U{{\mathcal U}}

 \def\NN{{\mathbb N}}

 \def\RR{{\mathbb R}}
 \def\SS{{\mathbb S}}
 
 \def\YY{{\mathbb Y}}
 
        \def\sspan{\operatorname{span}}
        
        \def\dim{\operatorname{dim}}

        \def\proj{\operatorname{proj}}

\def\({\Bigl(}
\def \){ \Bigr)}

\begin{document}

\title[ ]
{Spherical Harmonics}
\author{Feng Dai}
\address{Department of Mathematical and Statistical Sciences\\
University of Alberta\\, Edmonton, Alberta T6G 2G1, Canada.}
\email{dfeng@math.ualberta.ca}
\author{Yuan Xu}
\address{Department of Mathematics\\ University of Oregon\\
    Eugene, Oregon 97403-1222.}\email{yuan@math.uoregon.edu}
\thanks{The work of the first author was  supported  in part by NSERC  Canada under
grant RGPIN 311678-2010. The work of the second author was supported in part by
 NSF Grant DMS-1106113.}

\date{\today}

\begin{abstract}
This is Chapter 1 of the book {\it Approximation Theory and Harmonic Analysis on Spheres and Balls} by 
the authors. It provides a self-contained introduction to spherical harmonics. The book will be published 
as a title in {\it Springer Monographs in Mathematics} by Springer in 2013. The table of contents of the 
book is attached at the end of this file. 
\end{abstract}

\maketitle

In this chapter we introduce spherical harmonics and study their properties.
Most of the material of this chapter, except the last section, is classical. We strive
for a succinct account of the theory of spherical harmonics. After a standard
treatment of the space of spherical harmonics and orthogonal bases in the first
section, the orthogonal projection operator and reproducing kernels, also known as
zonal harmonics, are developed in greater details in  the second section, because of their
central role in harmonic analysis and approximation theory. As an application of
the addition formula, it is shown in the third section that there exist bases of
spherical harmonics consisting of entirely zonal harmonics. The Laplace--Beltrami
operator is discussed in the fourth section, where an elementary and self-contained
approach is adopted. Spherical coordinates and an explicit orthonormal basis
of spherical harmonics in these coordinates are presented the fifth section. These
formulas in two and three variables are collected in the sixth section for easy
reference, since they are most often used in applications. The connection to group
representation is treated briefly in the seventh section. The last section deals with
derivatives and integrals on the sphere. With the introduction of angular derivatives that are
first-order differential operators acting on the large circles of intersections of the
sphere and the coordinate planes, it is shown that the Laplace--Beltrami operator
can be decomposed into second-order angular derivatives. These derivative
operators will play an important role in approximation theory on the sphere.
They are used to derive several integral formulas on the sphere.

\section{Space of spherical harmonics and orthogonal bases} \label{sec:1.1}
\setcounter{equation}{0}

We begin by introducing some notation that will be used throughout this
book. For $x \in\RR^d$, we write $x =(x_1,\ldots, x_d)$. The inner product
of $x, y \in \RR^d$ is denoted by $\la x,y\ra : = \sum_{i=1}^d x_i y_i$ and
the norm of $x$ is denoted by $\|x\|:=\sqrt{\la x,x\ra}$. Let $\NN_0$ denote
the set of nonnegative integers. For $\a =(\a_1,\ldots,\a_d) \in \NN_0^d$,
a monomial $x^\a$ is a product $x^\alpha = x_1^{\alpha_1} \ldots x_d^{\alpha_d}$,
which has degree $|\alpha|=\a_1+\ldots+\a_d$.

A homogeneous polynomial $P$ of degree $n$ is a linear combination of
monomials of degree $n$, that is, $P(x) = \sum_{|\a| =n} c_\a x^\a$,
where $c_\a$ are either real or complex numbers. A polynomial
of (total) degree at most $n$ is of the form $P(x) = \sum_{|\a|\le n} c_\alpha x^\alpha$.
Let $\P_n^d$ denote the space of real  homogeneous polynomials of degree
$n$ and let $\Pi_n^d$ denote the space of real polynomials of degree at most $n$.
Counting the cardinalities  of $\{\a \in \NN_0^d: |\a| =n\}$ and $\{\a \in
\NN_0^d: |\a|  \le  n\}$ shows that 
\begin{equation*}
  \dim \P_n^d = \binom{n+d-1}{n} \quad\hbox{and} \quad \dim \Pi_n^d = \binom{n+d}{n}.
\end{equation*}

Let $\partial_i$ denote the partial derivative in the $i$-th variable
and $\Delta$ the Laplacian operator 
$$
       \Delta := \partial_1^2 + \cdots + \partial_d^2.
$$

\begin{defn} \label{defn:CHnd} 
For $n=0,1,2,\ldots $ let $\CH_{n}^{d}$ be the linear space of real
harmonic polynomials, homogeneous of degree $n$, on $\RR^d$, that
is,  
$$
   \CH_{n}^{d} :=\left\{ P\in \P_{n}^{d}:\Delta P=0\right\}.
$$
\end{defn}

Spherical harmonics are the restrictions of elements in $\CH_n^d$ to
the unit sphere. If $Y \in \CH_n^d$, then $Y(x) = \|x\|^n Y(x')$
where $x=\|x\|x'$ and $x ' \in \SS^{d-1}$. Strictly speaking, one
should make a distinction between $\CH_n^d$ and its restriction to
the sphere. We will, however, also call $\CH_n^d$ the space of
spherical harmonics. When it is necessary to emphasize the restriction
to the sphere, we shall use the notation $\CH_n^d\vert_{\sph}$.
In the same vein, we shall define $\P_n(\sph) := \P_n^d \vert_{\sph}$ and
$\Pi_n(\sph) := \Pi_n^d \vert_{\sph}$.  

Spherical harmonics of different degrees are orthogonal with respect to
\begin{equation}\label{ipd-sphere}
 \la f , g \ra_{\sph} := \frac{1}{\o_d} \int_{\sph} f (x) g(x)d\s(x)
\end{equation}
where $d\s$ is the surface area measure and $\o_d$ denotes the surface area of $\sph$,
\begin{equation}\label{surface-area}
 \o_d: = \int_{\sph} d\s = \frac{2 \pi^{d/2}} {\Gamma(d/2)}.
\end{equation}

\begin{thm} \label{thm:Hm-orthogonal}
If $Y_n \in \CH_n^d$, $Y_m \in \CH_m^d$, and $n \ne m$, then  
$
   \la Y_n , Y_m \ra_{\sph} = 0.
$
\end{thm}

\begin{proof}
Let $\frac{\partial}{\partial r}$ denote the normal derivative.
Since $Y_n$ is homogeneous, $Y_n(x) = r^n Y_n(x')$, where $x = r x'$ and
$x' \in \sph$, so that $\frac{\partial Y_n}{\partial r}(x') = n Y_n(x')$ for $x' \in \sph$
and $n \ge 0$. By Green's identity,
\begin{align*}
  (n - m)  \int_{\sph} Y_n Y_m d\s  & =   \int_{\sph} \( Y_m \frac{\partial Y_n}{\partial r}-
       Y_n \frac{\partial Y_m}{\partial r} \) d\s  \\
       & =   \int_{\ball} \(Y_m \Delta Y_n - Y_n \Delta Y_m \) dx =0,
\end{align*}
since $\Delta Y_n =0$ and $\Delta Y_m =0$.
\end{proof}

\begin{thm} \label{thm:Pn_decomp}
For $n =0,1,2,\ldots$, there is a decomposition of $\P_n^d$,
\begin{equation} \label{eq:Pn_decom1}
   \P_n^d = \bigoplus_{0 \le j \le n/2} \|x\|^{2j} \CH_{n-2j}^d.
\end{equation}
In other words, for each $P\in \P_n^d$, there is a unique decomposition
\begin{equation} \label{eq:Pn_decom2}
 P(x) = \sum_{0 \le j \le n/2} \|x\|^{2j} P_{n-2j}(x) \quad\hbox{with}\quad P_{n-2j} \in \CH_{n-2j}^d.
\end{equation}
\end{thm}

\begin{proof}
The proof uses induction. Evidently $\P_0^d = \CH_0^d$ and $\P_1^d
= \CH_1^d$. Since $\Delta \P_n^d \subset \P_{n-2}^d$, $\dim \CH_n^d
\ge \dim \P_n^d - \dim \P_{n-2}^d$. Suppose the statement holds
for $m= 0,1,\ldots, n-1$. Then $\|x\|^2 \P_{n-2}^d$ is a subspace
of $\P_{n}^d$ and it is isomorphic to $\P_{n-2}^d$. By the induction
hypothesis, $\|x\|^2 \P_{n-2}^d =  \bigoplus_{0 \le j \le n/2 -1}
\|x\|^{2j+2} \CH_{n-2-2j}^d$. Hence, by the previous theorem,
$\CH_n^d$ is orthogonal to $\|x\|^2 \P_{n-2}^d$, so that $\dim
\CH_n^d + \dim \P_{n-2}^d \le  \dim \P_{n}^d$. Consequently,
$\P_n^d = \CH_n^d \oplus \|x\|^2 \P_{n-2}^d$.
\end{proof}

\begin{cor}
For $n =0,1,2 , \ldots$, 
\begin{equation} \label{eq:dimHnd}
\dim \CH_{n}^{d}=\dim \P_{n}^{d}-\dim \P_{n-2}^{d}=
\binom{n+d-1}{n}-\binom{n+d-3}{n-2},
\end{equation}
where it is agreed that $ \dim \P_{n-2}^{d}=0$ for $n=0,1$.
\end{cor}

\begin{cor} \label{cor:dimPi-sph}  
For $n \in \NN$, $\Pi_n(\sph) = \P_n(\sph) \oplus \P_{n-1}(\sph)$ and
\begin{equation} \label{eq:dimPi-sph}
   \dim \Pi_n(\sph) = \dim \P_{n}^{d}+\dim \P_{n-1}^{d} = \binom{n+d-1}{n} + \binom{n+d-2}{n-1}.
\end{equation}
\end{cor}

\begin{proof}
By Theorem \ref{thm:Pn_decomp}, $\Pi_n(\sph)$ can be written as a direct sum
of $\CH_k^d$ for $0 \le k \le n$, which gives the stated decomposition
by \eqref{eq:Pn_decom1}. Moreover,
$$
  \dim \Pi_n(\sph) = \sum_{k=0}^n \dim \CH_k^d = \sum_{k=0}^n (\dim \P_k^d - \dim \P_{k-2}^d)
$$
by \eqref{eq:dimHnd}, which simplifies to \eqref{eq:dimPi-sph}.
\end{proof}

The orthogonality and homogeneity define spherical harmonics.

\begin{prop}
If $P$ is a homogeneous polynomial of degree $n$ and $P$ is orthogonal to
all polynomials of degree less than $n$ with respect to $\la \cdot, \cdot \ra_{\sph}$,
then $P \in \CH_n^d$.
\end{prop}

\begin{proof}
Since $P \in \P_n^d$, $P$ can be expressed as in \eqref{eq:Pn_decom2}. The orthogonality
then shows that $P = P_n \in \CH_n^d$.
\end{proof}

Let $O(d)$ denote the orthogonal group, the group of $d \times d$ orthogonal matrices, and
let $SO(d) = \{g \in O(d) : \det g = 1\}$ be the special orthogonal group. A rotation in $\RR^d$
 is determined by an element in $SO(d)$.  
 
\begin{thm} \label{cor:O(d)-basis}
The space $\CH_n^d$ is invariant under the action $f (x) \mapsto f (Qx)$, $Q \in O(d)$.
Moreover, if $\{Y_\a\}$ is an orthonormal basis of $\CH_n^d$, then so is $\{Y_\a(Q \{\cdot \})\}$.
\end{thm}

\begin{proof}
Since $\Delta$ is invariant under the rotation group $O(d)$ (writing $\Delta = \nabla \cdot \nabla$
and changing variables), if $Y \in \CH_n^d$ and $Q \in O(d)$ then $Y(Qx) \in \CH_n^d$.
That $\{Y_\a (Qx)\}$ is an orthonormal basis of $\CH_n^d$ whenever $\{Y_\a(x)\}$ is follows
from
$$
   \frac{1}{\o_d} \int_{\sph }Y_\a(Q x) Y_\b(Q x) d\s (x) =  \frac{1}{\o_d} \int_{\sph }Y_\a(x) Y_\b(x) d\s (x)
    = \delta_{\a,\b},
$$
which holds under a change of variables since $d\s$ is invariant under $O(d)$.
\end{proof}

Besides $\la f, g \ra_{\sph}$, another useful inner product can be defined on $\P_n^d$
through the action of differentiation. For $\a \in \NN_0^d$, let $\partial^\a:=\partial_1^{\a_1}\cdots\partial_d^{\a_d}$.
Let $(a)_n := a(a+1) \cdots (a+n-1)$ be the Pochhammer symbol.

\begin{thm} \label{thm:inner_partial}
For $p, q \in \P_n^d$, define a bilinear form
\begin{equation} \label{eq:inner_partial}
\la p, q \ra_\partial := p(\partial)q,
\end{equation}
where $p(\partial)$ is the differential operator defined by replacing $x^\a$ in $p(x)$ by
$\partial^\a$. Then
\begin{enumerate}[\quad 1.]
\item $\la p, q \ra_\partial$ is an inner product on $\P_n^d$;
\item the reproducing kernel of this inner product is $k_n (x, y) := \la x, y \ra^n /n!$;  
that is,
$$
 \la  k_n (x, \cdot), p \ra_\partial = p(x), \qquad \forall p \in
 \P_n^d;
$$
\item for $p\in \P_n^d$ and $q \in \CH_n^d$,
$$
  \la p, q \ra_\partial =  2^n \left(\frac{d}{2}\right)_n \la p,q\ra_{\sph}.
$$
\end{enumerate}
\end{thm}

\begin{proof}
Let $p, q \in \P_n^d$ be given by $p(x) = \sum_{|\a| =n} a_\a x^\a$ and $q(x) = \sum_{|\a| =n} b_\a x^\a$,
where $a_\a, b_\a \in \RR$. Then,
\begin{equation} \label{eq:inner_partial-2}
     \la p, q \ra_\partial = \sum_{|\a| =n} a_\a \partial^\a \sum_{|\b| =n} b_\b x^\b
         = \sum_{|\a| = n} \a! a_\a b_\a,
\end{equation}
which implies, in particular, that $\la p, p \ra_\partial > 0$ for $p \ne 0$. It follows then
that  $\la \cdot, \cdot \ra_\partial$ is an inner product on $\P_n^d$. By the
multinomial formula, for $q_\a(x) = x^\a$, $|\a| = n$,
$$
 \la k_n (x, \cdot), q_\a \ra_\partial = \frac{1}{n!}  \sum_{|\b|=n} \binom{n}{\b} x^\b \frac{\partial^\b}{\partial y^{\b}} y^\a
     = q_\a(x),
$$
which shows that $k_n (x, y)$ is the reproducing kernel with respect to $\la \cdot,\cdot\ra_\partial$.

We now prove item 3. Integrating by parts shows that
$$
 \int_{\RR^d} \partial_i f (x)g(x)e^{-\|x\|^2/2} dx = - \int_{\RR^d} f (x) (\partial_i g(x) - x_i g(x)) e^{-\|x\|^2/2} d x.
$$
Since $p(\partial)q$ is a constant, using this integration by parts repeatedly shows that 
\begin{align*}
 \la p, q \ra_\partial & = \frac{1}{(2 \pi)^{d/2}} \int_{\RR^d} p(\partial)q(x) e^{-\|x\|^2/2} dx  \\
  & = \frac{1}{(2\pi)^{d/2}} \int_{\RR^d} q(x)(p(x) + s(x))e^{-\|x\|^2/2} dx
\end{align*}
where $s \in \Pi_{n-1}^d$. Since $q \in \CH_n^d$ and $p \in \P_n^d$, switching to polar integral and using
the orthogonality of $\CH_n^d$, we obtain
$$
\la p, q \ra_\partial = \frac{1}{(2\pi)^{d/2} }\int_0^\infty r^{2n+d-1}  e^{-r^2/2} dr
  \int_{\sph} q(x')p(x')d\s(x').
$$
Evaluating the integral in $r$ and simplifying by \eqref{surface-area} concludes the proof.
\end{proof}

A large number of spherical harmonic polynomials can be defined explicitly through differentiation. Let
us denote the standard basis of $\RR^d$ by
$$
 e_1=(1,0,\cdots, 0), e_2= (0,1,0,\ldots, 0) \cdots, e_d=(0,\cdots, 0,1).
$$

\begin{thm}\label{thm:monic}
Let $d >2$. For $\a \in \NN_0^d$, $n = |\a|$, define
 
\begin{equation} \label{proj_monic}
p_\a (x) := \frac{(-1)^n}{ 2^n (\frac{d-2}{2})_n} \|x\|^{2 |\a|+d-2}  \partial^\a \left\{ \|x\| ^{- d+2} \right \}.
\end{equation}
Then
\begin{enumerate}[\quad 1.]
\item $p_\a \in \CH_n^d$ and $p_\a$ is the monic spherical harmonic of the form
\begin{equation} \label{monic_sh}
     p_\a(x) = x^\a + \|x\|^2 q_\a(x), \qquad q_\a \in \P_{n-2}^d.
\end{equation}
\item $p_\a$ satisfies the recurrence relation
\begin{equation} \label{monic_recu}
p_{\a+e_i} (x) = x_i p_\a(x) - \frac{1}{ 2n + d -2} \|x\|^2 \partial_i p_\a(x),
\end{equation}
\item  $\{p_\a : |\a| = n, \a_d = 0 \,\, \mathrm{or}\,\, 1\}$ is a basis of $\CH_n^d$.
\end{enumerate}
\end{thm}

\begin{proof}
Taking the derivative of $p_\a(x)$ gives immediately the recurrence relation \eqref{monic_recu}.
Clearly $p_0 (x) = 1$. By induction, the recurrence relation shows that $p_\a$ is a
homogeneous polynomial of degree $n$ and it is of the form \eqref{monic_sh}. We now show that
$p_\a$ is a spherical harmonic. For $g \in \P_n^d$
and $\rho \in \RR$, a quick computation using $\sum_{i=1}^d x_i \partial_i g(x) = n g(x)$ shows that
\begin{equation}\label{Delta_|x|g}
\Delta (\|x\|^\rho g) = \rho (2n + \rho + d - 2) \|x\|^{\rho -2} g + \|x\|^\rho \Delta g.
\end{equation}
In particular, setting $n = 0$ and $g(x) = 1$ gives $\Delta (\|x\|^{-d+2}) = 0$. Furthermore,
setting $g = p_\a$ and $\rho = -2n - d + 2$ in \eqref{Delta_|x|g} leads to
$$
\Delta p_\a (x) = \frac{(-1)^n}{2^n (d/2 - 1)_n} \|x\|^{2 |\a|+d-2}
  \partial^\a \Delta \{ \|x\|^{-d+2} \} = 0.
$$
Thus, $p_\a \in \CH_n^d$. Since $\|x\|^2 q(x)$ is a linear combination of the monomials $x^\b$
with $\b_d\ge 2$, by \eqref{monic_sh} and the linear independence of $\{x^\a: |\a| =n, \a_d=0
\ \text{or}\,1\}$, it follows that the elements in the set $\{p_\a: |\a| =n, \a_d=0  \ \text{or}\ 1\}$ are
linearly independent. The cardinality of the set is
$$
  \dim \P_n^{d-1} + \dim \P_{n-1}^{d-1} =  \binom{n+d-2}{d-2} + \binom{n+d-3}{d-2},
$$
which is, by a simple identity of binomial coefficients and \eqref{eq:dimHnd},  precisely $\dim \CH_n^d$.
This completes the proof.
\end{proof}

The right-hand side of \eqref{proj_monic} is called Maxwell's representation of harmonic
polynomials (\cite{Hob, Muller}). The complete set of $\{p_\a: |\a| =n\}$ is necessarily linearly
dependent by its cardinality. Moreover, by \eqref{proj_monic},
$$
   p_{\a+2 e_1} + \ldots + p_{\a+2 e_d} =  \frac{(-1)^n}{ 2^n (\frac{d}{2})_n} \|x\|^{2 |\a|+d-2}
      \partial^\a \Delta \left \{ \|x\| ^{- d+2} \right \} =0,
$$
which gives $\dim \P_{n-2}^d$ linearly dependent relations among $\{p_\a: |\a| =n\}$. The
set $\{p_\a: |\a| = n\}$ evidently contains many bases of $\CH_n^d$. The basis in  
item three of Theorem \ref{thm:monic} is but one convenient choice. The proof of Theorem
\ref{thm:monic} relies on the fact that $\|x\|^{- d+2}$ is a harmonic function in
$\RR^d\setminus \{0\}$ for $d > 2$. In the case of $d =2$, we need to replace this
function by $\log \|x\|$. Since the case $d =2$ corresponds to the classical Fourier
series, we leave the analogue of Theorem \ref{thm:monic} for $d =2$ to the interested
reader.

The basis $\{p_\a : |\a| = n, \a_d = 0 \,\, \mathrm{or}\,\, 1\}$ of $\CH_n^d$ is not orthonormal.
In fact, the elements of this basis are not mutually orthogonal. Orthonormal bases can be
constructed by applying the Gram-Schmidt process. An explicit orthonormal basis for
$\CH_n^d$ will be given in Section \ref{sec:spherical-coor} in terms of spherical coordinates.

\section{Projection operators and Zonal harmonics}\label{sec:projction}
\setcounter{equation}{0}

Let $L^2(\sph)$ denote the space of square integrable functions on $\sph$. Let
$$
   \proj_n : L^2 (\sph) \mapsto \CH_n^d
$$
denote the orthogonal projection from $L^2 (\sph)$ onto $\CH_n^d$. If $P \in \P_n^d$
then $P = P_n + \|x\|^2 Q_n$, where $P_n \in \CH_n^d$ and $Q_n \in \P_{n-2}^d$, by
\eqref{eq:Pn_decom2}, so that $\proj_n P = P_n$. In particular, \eqref{monic_recu} shows
that $p_\a$ defined in \eqref{proj_monic} is the orthogonal projection of the function
$q_\a(x) = x^\a$; that is, $p_\a = \proj_n q_\a$. This leads to the following:

\begin{lem} \label{thm:H-proj}
Let $p \in \P_n^d$. Then
\begin{equation} \label{eq:H-proj}
  \proj_n p = \sum_{j=0}^{\lfloor n/2 \rfloor} \frac{1}{4^j j !(-n + 2 - d/2)_j}
     \|x\|^{2j} \Delta^j p.
\end{equation}
\end{lem}

\begin{proof}
By linearity, it suffices to consider  $p$ being $q_\a (x) = x^\a$. By Theorem \ref{thm:monic},
$\proj_n q_\a(x)  = p_\a(x)$, and the proof amounts to showing that $p_\a(x)$ defined in
\eqref{proj_monic} can be expanded as in \eqref{eq:H-proj}. We use induction on $n$.
The case $n = 0$ is evident. Assume that \eqref{eq:H-proj}  has been established for
$m = 0, 1, \dots, n$. Applying \eqref{eq:H-proj} to $q_\a (x)$, $|\a| = n$, it follows that
\begin{align*}
\partial^\a \left \{ \|x\|^{-d+2} \right \} = & (-1)^n 2^n \left(\tfrac{d}2-1\right)_n \|x\|^{- 2n -d+2}  \\
  & \times \sum_{j=0}^{\lfloor n/2 \rfloor} \frac{1}{4^j j !(-n + 2 - d/2)_j} \|x\|^{2j}
   \Delta^j  \{x^\a\}.
\end{align*}
Applying $\partial_i$ to this identity, we obtain
\begin{align*}
  & \partial_i \partial^\a \left \{ \|x\|^{-d+2} \right \} = (-1)^n 2^n \left(\tfrac{d}2-1\right)_n (-2n - d + 2) \|x\|^{-2n -d+2} \\
& \qquad \times \sum_{j=0}^{\lfloor (n+1)/2 \rfloor} \frac{1}{4^j j !(-n +1 - d/2)_j} \|x\|^{2j}
   [ x_i \Delta^j  \{x^\a\} + 2j \Delta^{j-1} \partial_i \{x^\a \} ].
\end{align*}
The terms in the square brackets are exactly $\Delta^j \{x_i x^\a\}$
and the constant in front simplifies to $( -1)^{n+1}2^{n+1} \left(\tfrac{d}2-1\right)_{n+1}$,
so that the equation \eqref{eq:H-proj} holds for $p(x) = x_i x^\a$. This completes the induction.
\end{proof}

\begin{defn}  
The reproducing kernel $Z_n (\cdot,\cdot)$ of $\CH_n^d$ is uniquely
determined  by
\begin{equation}\label{zonal-defn}
\frac{1}{\o_d} \int_{\sph} Z_n (x, y) p(y)d\s(y) = p(x), \quad \forall p \in \CH_n^d, \quad x \in \sph
\end{equation}
and the requirement that $Z_n(x,\cdot)$ be an element of $\CH_n^d$
for each fixed $x$.
\end{defn}

That the kernel is well defined and unique follows from the Riesz representation
theorem applied to the linear functional $L(Y ) := Y (x)$, $Y\in \CH_n^d$, for a
fixed $x \in \sph$.

\begin{lem} \label{lem:zonal1}
In terms of an orthonormal basis $\{Y_j : 1 \le j \le \dim \CH_n^d\}$ of $\CH_n^d$,
\begin{equation} \label{zonal-defn2}
Z_n (x, y) = \sum_{k=1}^{\dim \CH_n^d} Y_k (x)Y_k (y), \quad x, y \in \sph,
\end{equation}
and, despite \eqref{zonal-defn2}, $Z_n$ is independent of the particular choice
of basis of $\CH_n^d$.
\end{lem}

\begin{proof}
Since $Z_n (x,\cdot) \in \CH_n^d$, it can be expressed as $Z_n (x,
y) =\sum_k c_k Y_k (y)$ where the coefficients are determined by
\eqref{zonal-defn} as $c_k = Y_k (x)$. The uniqueness implies that
$Z_n$ is independent of the choice of basis. This can also be
shown directly as follows. Let $\YY_n = (Y_1 , . . . , Y_N )$ with
$N = \dim  \CH_n^d$ and regard it as a column vector. Then $Z_n (x,
y) = [\YY_n (x)]^{\mathrm{tr}} \YY_n (y)$. If $\{Y_j' : 1 \le  j
\le N \}$ is another orthonormal basis of $\CH_n^d$, then $\YY_n' =
Q\YY_n$. Since the orthonormality of $\{Y_j \}$ can be expressed
as the fact that $\frac{1}{\o_d} \int_{\sph} \YY_n (x)[\YY_n (x)]^{\mathrm{tr}}
d\s(x)$ is an identity matrix, it follows readily that $Q$ is an
orthogonal matrix. Hence, $Z_n (x, y) = [\YY_n (x)]^{\mathrm{tr}}
Q^\mathrm{tr} Q \YY_n (y) = [\YY'_n (x)]^\mathrm{tr}\YY'_n (y)$.
\end{proof}

The reproducing  kernel is also the kernel for the projection operator.

\begin{lem} \label{Integral-proj}  
The projection operator can be written as
\begin{equation} \label{sp-proj-kernel}
\proj_n f (x) = \frac{1}{\o_d} \int_{\sph} f (y)Z_n (x, y)d\s(y).
\end{equation}
\end{lem}

\begin{proof}
Since $\proj_n f \in \CH_n^d$, it can be expanded in terms of the orthonormal
basis $\{Y_j , 1 \le j  \le N_n \}$, $N_n = \dim \CH_n^d$, of $\CH_n^d$, where the coefficients are
determined by the orthonormality,
$$
\proj_n f (x) = \sum_{j=1}^{N_n} c_j Y_j (x) \quad \hbox{with} \quad c_j
 = \frac{1}{\o_d} \int_{\sph}
   f (y)Y_j (y) d \s(y).
$$
If we pull out the integral in front of the sum, this is \eqref{sp-proj-kernel}
 by \eqref{zonal-defn2}.
\end{proof}

\begin{lem} \label{lem:zonal}
The kernel $Z_n (\cdot,\cdot)$ satisfies the following properties:
\begin{enumerate}[\quad 1.]
\item For every $\xi, \eta \in \sph$,
\begin{equation} \label{reproducing_Z}
 \frac{1}{\o_d} \int_{\sph} Z_n (\xi, y)Z_n (\eta,y)d \s(y) = Z_n (\xi, \eta).
\end{equation}
\item  $Z_n (x, y)$ depends only on $\la x, y\ra$.
\end{enumerate}
\end{lem}

\begin{proof}
By Corollary \ref{cor:O(d)-basis}, the uniqueness of $Z_n (x, y)$ shows that
 $Z_n (Qx, Qy) = Z_n (x, y)$ for all $Q \in  O(d)$. Since for $x, y \in \sph$ there exists a $Q
\in SO(d)$ such that $Qx = (0, . . . , 0, 1)$ and $Qy = (0, . . . , 0, \sqrt{1- \la x, y\ra^2}, \la x, y \ra)$,
this shows that $Z_n (x, y)$ depends only on $\la x, y\ra$.
\end{proof}

From the second property of the lemma, $Z_n (x, y) = F_n ( \la x, y \ra)$, which is
often called a zonal harmonic, since it is harmonic and depends only on $\la x, y\ra$.
We now derive a closed formula for $F_n$, which turns out to be a multiple of the
Gegenbauer polynomial, $C_n^\l$, of degree $n$ defined, for $\l > 0$ and
$n \in \NN_0$, by 
\begin{equation} \label{chapt1-Gegen-2F1}
  C_{n}^{\lambda }(x) :=\frac{(\lambda )_{n}2^{n}}{n!}x^{n}{}_{2}F_{1}\left(
       \begin{matrix} -\frac{n}{2},\frac{1-n}{2}\cr 1-n-\lambda \end{matrix};
          \frac{1}{x^{2}}\right),
\end{equation}
where ${}_2F_1$ is the hypergeometric function. The properties of the Gegenbauer
polynomials are collected in Appendix B.

\begin{thm}\label{thm:zonal}
For $n \in \NN_0$ and $x, y \in \sph$, $d \ge 3$,
\begin{equation} \label{zonal}
Z_n (x, y) =  \frac{n + \l} {\l} C_n^\l( \la x, y\ra ), \quad \l = \frac{d-2}{2}.
\end{equation}
\end{thm}

\begin{proof}
Let $p \in \CH_n^d$. By Theorem \ref{thm:inner_partial}, $p(x) = \la k_n (x, \cdot), p \ra_\partial$. For fixed $x$, it
follows from the same theorem that
\begin{align*}
p(x) &  = \la k_n (x, \cdot), p \ra_\partial = \la \proj_n
 (k_n (x, \cdot)), p\ra_\partial \\
  & = \frac{2^n (d/2)_n}{\o_d}  \int_{\sph} \proj_n [k_n (x, \cdot)](y) p(y) d\s (y).
\end{align*}
Since the kernel $Z_n (\cdot, \cdot)$ is uniquely determined by
the reproducing property, this shows that $Z_n (x, y) = 2^n
(\frac{d}{2})_n \proj_n [k_n (x, \cdot)](y)$. Since $k_n (x,
\cdot)$ is a homogeneous polynomial of degree $n$ and, taking
the derivative on $y$, we have $\Delta^j k_n (x, y) = \|x\|^{2j}k_{n -2j} (x,
y)$, as is easily seen from $\partial_i k_n (x, y)
 = x_i k_{n -1}(x, y)$, Lemma \ref{thm:H-proj} shows, for $x, y \in \sph$, that
 $$
Z_n (x, y) = 2^n \( \frac{d}{2} \)_n \proj_n [k_n (x, \cdot)](y) = \sum_{j =0}^{\lfloor n/2\rfloor}
 \frac{ (\frac{d}{2})_n 2^{n -2j} }{ j! (1-n - \l)_j} k_{n-2j} (x, y).
$$
Using the fact $1/(n - 2j )! = (-n)_{2j} /n! = 2^{2j}(- \frac{n}2)_j (\frac{-n+1}2 )_j /n!$, we conclude then
\begin{align*}
Z_n (x, y) & = \frac{n + \l}{\l} \frac{(\l)_n 2^n}{n!}  \sum_{j =0}^{\lfloor n/2\rfloor}
  \frac{ (- \frac{n}2)_j (\frac{-n+1}2 )_j }{j!(1 - n -\l)_j} \la x, y \ra^{n -2j} \\
    & =  \frac{n + \l}{\l} \frac{(\l)_n 2^n}{n!}  \la x, y \ra^{n}{}_{2}F_{1}\left(
 \begin{matrix} -\frac{n}{2},\frac{1-n}{2}\cr 1-n-\lambda \end{matrix};
\frac{1}{ \la x, y \ra^{2}}\right),
\end{align*}
from which the stated result follows from \eqref{chapt1-Gegen-2F1}.
\end{proof}

Let $\{Y_i : 1 \le i \le \dim \CH_n^d \}$ be an orthonormal basis of $\CH_n^d$. Then \eqref{zonal} states that 
 \begin{equation}\label{addition_formula}
\sum_{ j=1}^{\dim \CH_n^d}
Y_j (x)Y_j (y)   =\frac{n + \l}{\l}C_n^\l (\la x, y \ra), \qquad \l = \frac{d-2}{2}.
\end{equation}
This identity is usually referred to as the addition formula of spherical harmonics, since
for $d = 2$ it is the addition formula of the cosine function (see
 Section \ref{sec:sph-harmonic-d=3}).

\begin{cor} \label{cor:zonal}
For $n \in \NN_0$ and $x, y \in \sph$, $d \ge 3$,
\begin{equation} \label{eq:cor:zonal}
        |Z_n(x,y)| \le \dim \CH_n^d \quad\hbox{and}\quad Z_n(x,x) = \dim \CH_n^d.
\end{equation}
\end{cor}

\begin{proof}
Set $F_n(t): = \frac{n+\l}{\l} C_n^\l(t)$. By \eqref{zonal}, $Z_n(x,x) = F_n(1)$ is a
constant for all $x \in \sph$. Setting $x = y$ in \eqref{zonal-defn2} and integrating over
$\sph$, we obtain
$$
 F_n(1) = \frac{1}{\o_d} \int_{\sph} Z_n(\la  x,x \ra) d\s(x)
   = \frac{1}{\o_d} \int_{\sph}  \sum_{k=1}^{\dim \CH_n^d} Y_k^2 (x) d\s(x)
   = \dim \CH_n^d.
$$
The inequality follows from applying the Cauchy-Schwarz inequality to \eqref{zonal-defn2}.
\end{proof}

Because of the relation \eqref{zonal}, the Gegenbauer polynomials with $\l = \frac{d-2}{2}$
are also called ultraspherical polynomials. A number of properties of the Gegenbauer
polynomials can be obtained from the zonal spherical harmonics. For example, the corollary
implies that $C_n^{\l}(1) = \frac{\l}{n+\l} \dim \CH_n^d$. Here is another example:

\begin{cor}
For $\l = \frac{d-2}{2}$, the Gegenbauer polynomials $C_n^\l$ satisfy the orthogonality
relation
\begin{equation} \label{eq:ultrasphericalOP}
    \frac{\o_{d-1}}{\o_d} \int_{-1}^1 C_n^\l (t) C_m^\l(t) (1-t^2)^{\l - \frac12} dt =h_n^\l \delta_{m,n},
\end{equation}
where
$$
   h_n^\l  =  \frac{\l}{n+\l} C_n^\l(1).
$$
\end{cor}

\begin{proof}
Set again $F_n(t) = \frac{n+\l}{\l} C_n^\l (t)$. Since $Z_n(x,\cdot)$ and $Z_m(x,\cdot)$
are orthogonal over $\sph$, and by (A.5.1), their integrals can be written as
an integral of one variable, we obtain
\begin{align*}
   \frac{\o_{d-1}}{\o_d} \int_{-1}^1  F_n (t)F_m(t)(1 - t^2)^{\frac{d -3}{2}} d t
   & = \frac{1}{\o_d} \int_{\sph} F_n (\la x, y \ra)F_m(\la x, y\ra )d\s(y) \\
   & = F_n(1) \delta_{m,n} = (\dim \CH_n^d) \delta_{m,n},
\end{align*}
where the second line follows from \eqref{reproducing_Z} and Corollary \ref{cor:zonal}.
\end{proof}

The functions on $\sph$ that depend only on $\la x, y\ra$ are analogues of radial
functions on $\RR^d$ . For such functions, there is a Funk-Hecke formula given below.

\begin{thm} \label{thm:Funk-Hecke}  
Let $f$ be an integrable function such that  $\int_{-1}^1 |f(t)| (1-t^2)^{(d-3)/2} dt$ is finite
and $d \ge 2$. Then for every $Y_n\in \CH_n^d$,
\begin{equation}\label{eq:funk-hecke}
\int_{\sph} f (\la x, y\ra)Y_n (y)d\s(y) = \l_ n (f )Y_n (x),
\quad x \in \sph,
\end{equation}
where $\l_n (f )$ is a constant defined by
$$
\l_n (f ) = \o_{d-1} \int_{-1}^1 f (t) \frac{C_n^{\frac{d-2}{2}}(t) } {C_n^{\frac{d-2}{2}}(1)}
 (1 - t^2)^{\frac{d-3}{2}} dt.
$$
\end{thm}

\begin{proof}
If $f$ is a polynomial of degree $m$, then we can expand $f$ in terms of
the Gegenbauer polynomials
$$
f (t) = \sum_{k=0}^m  \l_k \frac{k + \frac{d-2}{2}}{\frac{d-2}{2}}
C_k^{\frac{d-2}{2}}(t).
$$
where $\l_k$ are determined by the orthogonality of the Gegenbauer polynomials,
$$
\l_k = \frac{c_d}{C_k^{\frac{d-2}{2}}(1)}  \int_{-1}^1 f(t)
C_k^{\frac{d-2}{2}}(t)
  (1 - t^2)^{\frac{d-3}{2}} dt,
$$
and $c_d^{-1} = \int_{-1}^1 (1 -  t^2)^{\frac{d-3}{2}} dt  =
\o_d/\o_{d-1}$. From \eqref{zonal} and the reproducing property
of $Z_n (x, y)$, it follows that for $n\le m$
$$
\frac{1}{\o_d} \int_{\sph} f (\la x,y\ra) Y_n (y)d\s(y) = \l_n Y_n
(x), \quad x \in \sph.
$$
Since $\l_n /\o_d = \l_n (f )$ by definition, we have established the Funk-Hecke
formula \eqref{eq:funk-hecke} for polynomials, and hence, by the Weierstrass theorem,
for continuous functions, and the function satisfying the integrable condition in the
statement can be approximated by a sequence of continuous functions.
\end{proof}

\section{Zonal basis of spherical harmonics} \label{sec:FundamentalSystem}
\setcounter{equation}{0}

In view of the addition formula of spherical harmonics, one may ask whether there is a basis of
spherical harmonics that consists entirely of zonal harmonics. This question is closely related to
the problem of interpolation on the sphere.


Throughout this subsection, we fix $n$, set $N = \dim \CH_n^d$ and fix $\{Y_1,\ldots, Y_N\}$
as an orthonormal basis for $\CH_n^d$. Let $\{x_1, \ldots, x_N\}$ be a collection of points on
$\sph$.  We let $M_1 := Y_1(x_1)$ and, for $k =2,3,\ldots, N$, define matrices
$$
M_k:= \left[ \begin{matrix} Y_1(x_1) &\ldots & Y_{1}(x_k) \\
                                           \vdots & \ldots & \vdots \\
                                         Y_k(x_1) &\ldots & Y_k(x_k)  \end{matrix} \right] , \quad
M_k(x) := \left[ \begin{array}{c|c}
  M_{k-1}  & \begin{matrix} Y_1(x) \\   \vdots \\ Y_{k-1}(x) \end{matrix}\\
  \hline Y_{k}(x_1),  \hdots Y_{k}(x_{k-1}) & Y_{k}(x),
      \end{array} \right].
$$
The product of $M_N$ and its transpose $M_N^T$ can be summed on applying the addition
formula \eqref{addition_formula}, as $M_N^T M_N = [Z_n(x_i,x_j)]_{i,j=1}^N$, which shows, in particular,
\begin{equation} \label{eq:detM}
\det  [Z_n(x_i,x_j)]_{i,j=1}^N = (\det M_N)^2 \ge 0.
\end{equation}
This motivates the following definition.

\begin{defn} 
A collection of points $\{x_1,\ldots,x_N\}$ in $\sph$ is called a fundamental system of degree $n$ on
the sphere $\sph$ if
$$
   \det \left[ C_n^\l( \la x_i, x_j\ra)\right]_{i,j =1}^N > 0, \qquad \l = \tfrac{d-2}{2}.
$$
\end{defn}

\begin{lem}
There exists a fundamental system of degree $n$ on the sphere.
\end{lem}

\begin{proof}
The existence of a fundamental system follows from the linear independence of $\{Y_1,\ldots, Y_N\}$.
Indeed, we can clearly choose $x_1 \in \sph$ such that $\det M_1 = Y_1(x_1) \ne 0$. Assume that
$x_1,\ldots, x_k$, $1 \leq  k \le N-1$, have been chosen such that $\det M_k \ne 0$. The determinant
$\det M_{k+1}(x)$ is a polynomial of $x$ and cannot be identically zero by the linear independence
of $\{Y_1,\ldots, Y_{k+1}\}$, so that there is a $x_{k+1} \in \sph$ such that $\det M_{k+1} = \det M_k(x_{k+1}) \ne 0$.
In this way, we end up with a collection of points $\{x_1,\ldots, x_N\}$ on $\sph$ that satisfies $\det M_N \ne 0$,
which implies $\det  \left[ C_n^\l( \la x_i, x_j\ra)\right]_{i,j =1}^N  = c_N (\det M_N)^2 > 0$, where $c_N = \l^N /(n+\l)^N$.
\end{proof}

The proof shows, in fact, that there are infinitely many fundamental systems. Indeed, regarding
$x_1,\ldots, x_N$ as variables, we say that $\det M_N$ is a $(d-1)N$-dimensional polynomial in these variables,
and its zero set is an algebraic surface of $\RR^{(d-1)N}$, which necessarily has measure zero.

\begin{thm}
If $\{x_1,\ldots, x_N\}$ is a fundamental system of points on the sphere, then
$\{C_n^{\l}(\la \cdot, x_i \ra ): i =1,2,\ldots, N\}$, $\l=\frac{d-2}{2}$, is a basis of $\CH_n^d |_{\sph}$.
\end{thm}

\begin{proof}
Let $P_i(x) = \frac{n+\l}{\l}C_n^{\l}(\la \cdot, x_i \ra)$. The addition theorem \eqref{addition_formula} gives
$$
   P_i(x) = \sum_{k=1}^N Y_k(x_i) Y_k(x), \qquad  i =1,2, \ldots, N,
$$
which shows that $\{P_1,\ldots, P_N\}$ is expressed in the basis $\{Y_1,\ldots, Y_N\}$ with 
transition matrix given by $M_N = [Y_k(x_i)]_{k,i=1}^N$. Since $\{x_1,\ldots, x_N\}$ is fundamental,
the matrix is invertible, by \eqref{eq:detM}. We can then invert the system to express $Y_k$ as a
linear combination of $P_1,\ldots, P_N$, which completes the proof.
\end{proof}

A word of caution is in order. The polynomial $C_n^{\lambda}(\la x,  x_i \ra)$ is, for $x \in \sph$, a linear
combination of the spherical harmonics according to the addition formula. It is not, however, a
homogeneous polynomial of degree $n$ in $x \in \RR^d$; rather, it is the restriction of the homogeneous
polynomial $\|x\|^n C_n^{\l}(\la x/\|x\|, y\ra)$ to the sphere. This is a situation in which the distinction between
$\CH_n^d$ and $\CH_n^d |_{\sph}$ is called for; see the discussion below Definition \ref{defn:CHnd}.

Fundamental sets of points are closely related to the problem of interpolation. Indeed, it can be
stated as follows: for a given set of data $\{(x_j,y_j): 1 \le j \le N\}$, $x_j \in \sph$ and $y_j \in \RR$,
there is a unique element $Y \in  \CH_n^d$ such that $Y (x_j) = y_j$, $j=1,\ldots, N$, if and only if
the points $\{x_1,\ldots,x_N\}$ form a fundamental system on the sphere. Much more interesting
and challenging is the problem of choosing points in such a way that the resulting basis of zonal spherical
harmonics has a relatively simple structure.

A related result is a zonal basis for the space $\P_n$ of homogeneous polynomials of degree
$n$ and the space $\Pi_n^d$ of all polynomials of degree at most $n$.

\begin{thm} \label{thm:zonal-basis}  
There exist points $\xi_{j,n} \in \sph$, $1 \le j \le r_n^d : = \dim \P_n^d$, such that
\begin{enumerate}[$(i)$]
\item$\{\la x ,\xi_{j,n}\ra^n: 1 \le j \le r_n^d\}$ is a basis for $\P_n^d$ of degree $n$.

\item For each polynomial $f \in \Pi_n^d$, there exist polynomials $p_j: [-1,1] \mapsto \RR$ for
$1 \le j \le r_n^d$ such that
$$
    f(x) = \sum_{j=1}^{r_n^d} p_j (\la x ,\xi_{j,n}\ra).
$$
\end{enumerate}
\end{thm}

\begin{proof}
Following the proof of the existence of fundamental system of points, it is easy to see
that there exist points $\xi_{j,n}$ such that $f(\xi_{j,n}) = 0$ if and only if $f =0$ for all
$f \in \P_n^d$. For the proof of (i), we first deduce by the binomial formula that
$$
  f_{j,n} (x): = \la x ,\xi_{j,n} \ra^n = \sum_{|\a| =n} \f {n!}{\a!} \xi^{\a}_{j,n} x^\a, \quad 1 \le j \le r_n^d.
$$
Let $f \in \P_n^d$. Then $f(x) = \sum_{|\a|=n} a_\a x^\a$. By \eqref{eq:inner_partial-2},
$$
   \la f, f_{j,n} \ra_\partial = n! \sum_{|\a| =n} a_\a \xi^{\a}_{j,n} = n! f (\xi_{j,n}),
$$
which implies, by the choice of $\xi_{j,n}$, that $ \la f, f_{j,n} \ra_\partial = 0$, $1 \le j \le r_n^d$,
if and only if $f = 0$. Thus, $\{f_{j,n}\}^\perp = \{0\}$. Since $f_{j,n} \in \P_n^d$. This proves (i).

For the proof of (ii), let $m$ be an integer, $0 \le m \le n -1$, and let $f_{j,m,n} := \la x ,\xi_{j,n} \ra^m$.
For $f \in \P_m^d$, it follows as above that $\la f , f_{j,m,n} \ra_\partial = m! f(\xi_{j,m})$,
$1 \le j \le r_n^d$. Hence, if $\la f , f_{j,m,n} \ra_\partial  = 0$ for $1 \le j \le r_n^d$, then
$f(\xi_{j,m}) = (f g )(\xi_{j,m})= 0$ for  $1 \le j \le r_n^d$ and all $g \in \P_{n-m}^d$, which
implies by the choice of $\xi_{j,n}$ and the fact that $f g \in \P_n^d$, that $f g = 0$ or $f =0$.  Consequently,
$\P_m^d =\sspan\{f_{j,m,n}: 0 \le j \le r_n^d\}$ for $0 \le m \le n$. Since $\Pi_n^d = \sum_{m =0}^n
\P_m^d$, this proves (ii).
\end{proof}

\section{Laplace-Beltrami operator} \label{sec:Laplace-Beltrami}  
\setcounter{equation}{0}

The operator in the section heading is the spherical part of the Laplace operator,
which we denote by $\Delta_0$. The operator $\Delta_0$ plays an important role
for analysis on the sphere. The usual approach to deriving this operator relies on
an expression of the Laplace operator under a change of variables, which we
describe first.

For $x \in \RR^d$ let $x \mapsto u = u(x)$ be a change of variables that is a bijection,
so that we can also write $x = x(u)$. Introduce  the tensors
$$
    g_{i, j} := \sum_{k=1}^d \frac{\partial x_k}{\partial u_i}\frac{\partial x_k}{\partial u_j}
      \quad \hbox{and}\quad
     g^{i,j } := \sum_{k=1}^d \frac{\partial u_i}{\partial x_k}\frac{\partial u_j}{\partial x_k},  \quad
      1 \le i, j\le d,
$$
and let $g := \det (g_{i,j})_{i,j=1}^d$. Then $(g_{i,j})^{-1} = (g^{i,j})$. A general result in
Riemannian geometry, or a bit of tensor analysis, shows that the Laplace operator satisfies
\begin{equation} \label{Beltrami-general}
    \Delta = \sum_{i=1}^d  \frac{\partial^2}{\partial x_i^2} =
       \frac{1}{\sqrt{g}}  \sum_{i=1}^d \sum_{j=1}^d \frac{\partial}{\partial u_i}  \sqrt{g} g^{i,j}
              \frac{\partial}{\partial u_j}.
\end{equation}
The Laplace--Beltrami operator, i.e., the spherical part of the Laplace operator, can then be
derived from \eqref{Beltrami-general} by the change of variables $x \mapsto (r, \xi_1,\ldots,\xi_{d-1})$,
where $r > 0$ and $\xi = (\xi_1,\ldots, \xi_d) \in \sph$. For this approach and a derivation of
\eqref{Beltrami-general}, see \cite{Muller}. We shall adopt an approach that is elementary
and self-contained.

\begin{lem} \label{lem:Laplace-Beltrami}
In the spherical--polar coordinates $x = r \xi$, $r > 0$, $\xi \in \sph$,
the Laplace operator satisfies
\begin{equation}\label{Beltrami-radial}
\Delta = \frac{\partial^2}{\partial r^2} + \frac{d - 1}{r} \frac{\partial}{\partial r} + \frac{1}{r^2}\Delta_0,
\end{equation}
where  
\begin{equation}\label{Beltrami-explicit}
    \Delta_0 = \sum_{i=1}^{d-1} \frac{\partial^2} {\partial \xi_i^2} - \sum_{i=1}^{d-1} \sum_{j=1}^{d-1}
       \xi_i\xi_j  \frac{\partial^2} {\partial \xi_i\partial\xi_j} -(d-1) \sum_{i=1}^{d-1} \xi_i  \frac{\partial} {\partial \xi_i}.
\end{equation}
\end{lem}

\begin{proof}
Since $\xi \in \sph$, we have $\xi_1^2+\ldots + \xi_d^2 =1$. We evaluate the Laplacian $\Delta$ under
a change of variables $(x_1,\ldots,x_d) \mapsto (r,\xi_1,\ldots, \xi_{d-1})$ under $x = r \xi$, which has
inverse $\xi_1 = x_1/\|x\|, \ldots, \xi_{d-1} = x_{d-1} /\|x\|, r = \|x\|$.  The chain rule leads to
\begin{align} \label{eq:Beltrami-radial-2}
\begin{split}
 \frac{\partial} {\partial x_i} & =  \frac{1}{r}  \frac{\partial} {\partial \xi_i} - \frac{\xi_i}{r}
     \sum_{j=1}^{d-1} \xi_j  \frac{\partial} {\partial \xi_j} + \xi_i \frac{\partial} {\partial r}, \quad 1 \le i \le d-1, \\
    \frac{\partial} {\partial x_d} & =
     - \frac{x_d}{r^2} \sum_{j=1}^{d-1} \xi_j  \frac{\partial} {\partial \xi_j} +\frac{x_d}{r} \frac{\partial} {\partial r}.
\end{split}
\end{align}
If we apply the product rule for the partial derivative on $x_d$, it follows that
\begin{align*}
  \Delta = & \sum_{i=1}^{d-1} \bigg(\frac{1}{r}  \frac{\partial} {\partial \xi_i} - \frac{\xi_i}{r}
     \sum_{j=1}^{d-1} \xi_j  \frac{\partial} {\partial \xi_j} + \xi_i \frac{\partial} {\partial r}\bigg)^2 \\
           & + \xi_d^2 \bigg( - \frac{1}{r} \sum_{j=1}^{d-1} \xi_j  \frac{\partial} {\partial \xi_j}+\frac{\partial} {\partial r}
             \bigg)^2 +(1-\xi_d^2)\biggl( - \frac{1}{r^2} \sum_{j=1}^{d-1} \xi_j  \frac{\partial} {\partial \xi_j} +\frac{1}{r} \frac{\partial} {\partial r}\biggr),
\end{align*}
where we have used $x_d =  r\xi_d$, from which a
straightforward, though tedious, computation, and
simplification using $\xi_1^2+ \ldots + \xi_d^2 =1$, establishes \eqref{Beltrami-radial} and \eqref{Beltrami-explicit}.
\end{proof}

The Laplace--Beltrami operator also satisfies  a recurrence relation that can be used to
derive an explicit formula for $\Delta_0$ under a given coordinate system of $\sph$.
We write $\Delta_{0,d}$ instead of $\Delta_0$ when we need to emphasize the dimension.

\begin{lem} 
Let $\Delta_{0,d}$ be the Laplacian--Beltrami operator for $\sph$. For $\xi \in \sph$, write
$\xi = (\sqrt{1-t^2} \eta, t)$ with $-1 \le t \le 1$ and $\eta \in \SS^{d-2}$. Then
\begin{align} \label{Beltrami-iterate}
   \Delta_{0,d} = \frac{1}{(1-t^2)^{\frac{d-3}{2}} }\frac{\partial}{\partial t}\left (
           (1-t^2)^{\frac{d-1}2} \frac{\partial}{\partial t} \right) + \frac{1}{1-t^2} \Delta_{0,d-1}.
\end{align}
\end{lem}

\begin{proof}
We work with the expression for $\Delta_{0,d}$ in \eqref{Beltrami-explicit} and make a change of
variables $(\xi_1,\ldots,\xi_{d-1}) \mapsto (\eta_1,\ldots,\eta_{d-2}, t)$ defined by
$$
     \xi_1 = \sqrt{1-t^2} \eta_1, \ldots, \xi_{d-2} = \sqrt{1-t^2} \eta_{d-2}, \quad \xi_{d-1} = t,
$$
where we have switched $\xi_{d-1}$ and $\xi_d$ for convenience. The chain rule gives
$$
    \frac{\partial}{\partial \xi_i} = \frac{1}{\sqrt{1-t^2} } \frac{\partial}{\partial \eta_i}, \quad 1 \le i \le d-2,\quad
    \hbox{and} \quad   \frac{\partial}{\partial \xi_{d-1} } =
         \frac{t}{1-t^2} \sum_{j=1}^{d-2} \eta_j \frac{\partial}{\partial \eta_j}  + \frac{\partial}{\partial t},
$$
which can be used iteratively to compute  $\Delta_{0,d}$ on writing \eqref{Beltrami-explicit} as
\begin{align*}
   \Delta_{0,d} =(1-t^2)   \frac{\partial^2}{\partial \xi_{d-1}^2} & + \sum_{i=1}^{d-2}   \frac{\partial^2}{\partial \xi_i^2}
       - 2 t \sum_{i=1}^{d-2}  \xi_i   \frac{\partial^2}{\partial \xi_i \partial\xi_{d-1}} \\
     & - \sum_{i=1}^{d-2}\sum_{j=1}^{d-2} \xi_i\xi_j \frac{\partial^2}{\partial \xi_i \partial\xi_j}
     - (d-1) \sum_{i=1}^{d-1} \frac{\partial}{\partial \xi_i}.
\end{align*}
A straightforward computation and another use of \eqref{Beltrami-explicit} then leads to
$$
  \Delta_{0,d} = (1-t^2) \frac{\partial^2}{\partial t^2} - (d-1) t \frac{\partial}{\partial t} + \frac{1}{1-t^2} \Delta_{0,d-1},
$$
which is precisely \eqref{Beltrami-iterate}.
\end{proof}

The formula \eqref{Beltrami-explicit} gives an explicit expression for $\Delta_0$ in the local
coordinates of $\sph$. An explicit formula for $\Delta_0$ in terms of spherical coordinates will be
given in Section \ref{sec:spherical-coor}.

Let $\nabla = (\partial_1,\ldots,\partial_d)$. The proof of Lemma \ref{lem:Laplace-Beltrami} also shows that
\begin{equation} \label{eq:nabla-radial}
    \nabla =  \frac{1}{r} \nabla_0  +  \xi \frac{\partial}{\partial r}, \qquad x = r \xi, \quad \xi \in \sph,
\end{equation}
where $\nabla_0$ is the spherical gradient, which is the spherical part of $\nabla$ and involves only derivatives in $\xi$. Its
explicit expression can be read off from \eqref{eq:Beltrami-radial-2}. We shall not need this expression
and will be content with the following expression.  

\begin{cor}
Let $f \in C^2(\sph)$. Define $F(y):= f(y / \|y\|)$, $y \in \RR^d$. Then
\begin{equation} \label{Beltrami}
     \Delta_0 f (x) = \Delta F(x) \quad \hbox{and} \quad \nabla_0 f(x) = \nabla F(x), \quad x \in \sph.
\end{equation}
\end{cor}

The corollary follows immediately from \eqref{Beltrami-radial} and \eqref{eq:nabla-radial}, since $x/\|x\|$ is
independent of $r$. The expressions in \eqref{Beltrami} show that $\Delta_0$ and $\nabla_0$ are independent
of the coordinates of $\sph$. In fact, we could take \eqref{Beltrami} as the definition of $\Delta_0$ and $\nabla_0$.

The usual Laplacian $\Delta$ can be expressed in terms of the dot product of $\nabla$, $\Delta = \nabla \cdot \nabla$,
which can also be written--as is often in physics textbooks--as $\Delta = \nabla^2$. The analogue of this identity
also holds on the sphere.

\begin{lem} \label{Delta=nabla2}
The Laplace--Beltrami operator satisfies
\begin{equation} \label{eq:Delta=nabla2}
    \Delta_0 = \nabla_0 \cdot \nabla_0.
\end{equation}
\end{lem}

\begin{proof}
An application of \eqref{eq:nabla-radial} gives immediately
\begin{align*}
  \Delta = \nabla \cdot \nabla = \frac{1}{r^2} \nabla_0 \cdot \nabla_0 + \frac{1}{r} \nabla_0 \left(\xi \frac{\partial}{\partial r} \right)
      + \xi \frac{\partial}{\partial r} \left (\frac 1 r \nabla_0 \right) +  \frac{\partial^2 }{\partial r^2}.
\end{align*}
We note that $\xi \cdot \nabla_0 f(\xi) = 0$, since $\xi \in \sph$ is in the normal direction of $\xi$ whereas $\nabla_0 f(\xi)$,
by \eqref{eq:nabla-radial}, is on the tangent plane at $\xi$. Hence, we see that
$$
   \xi \frac{\partial}{\partial r} \left (\frac 1 r \nabla_0 \right) = - \frac{1}{r^2} \xi \cdot \nabla_0 + \frac{1}{r} \xi \cdot
     \nabla_0 \frac{\partial}{\partial r} = 0.
$$
Using \eqref{eq:nabla-radial}, a quick computation gives $\nabla_0 \cdot \xi = d-1$, so that by the product rule,
$$
 \nabla_0 \left(\xi \frac{\partial}{\partial r} \right) = \nabla_0 \cdot \xi \frac{\partial}{\partial r}  +
     \frac{\partial}{\partial r} \xi \cdot \nabla_0 = (d-1)\f {\partial}{\partial r}.
$$
Consequently, we conclude that
$$
    \Delta = \frac{\partial^2 }{\partial r^2} + \frac{d-1}{r} \frac{\partial}{\partial r} + \frac{1}{r^2} \nabla_0 \cdot \nabla_0.
$$
Comparing this with \eqref{Beltrami-radial} completes the proof.
\end{proof}

Our next result shows that the spherical harmonics are eigenfunctions of the Laplace-Beltrami operator.

\begin{thm} \label{thm:eigen-eqn-sph}  
The spherical harmonics are eigenfunctions of $\Delta_0$,
\begin{equation}\label{eigen-eqn-sph}
\Delta_0 Y (\xi) = - n(n + d - 2)Y(\xi),  \quad \forall Y  \in \CH_n^d, \quad \xi \in \sph.
\end{equation}
\end{thm}

\begin{proof}
Let $x = r \xi$, $\xi \in \sph$. Since $Y \in \CH_n^d$ is homogeneous, $Y (x) = r^ n Y (\xi)$ and 
by \eqref{Beltrami-radial},
$$
0 = \Delta Y(x) = n(n -1)r^{n -2} Y (\xi) + (d - 1)n r^{n -2} Y (\xi) + r^{n-2} \Delta_0 Y (\xi),
$$
which is \eqref{eigen-eqn-sph} upon dividing by $r^{n-2}$.
\end{proof}

The identity \eqref{eigen-eqn-sph} also implies that $\Delta_0$ is self-adjoint, which can also be proved
directly and will be treated in the last section of this chapter, together with a number of other properties of the
Laplace-Beltrami operator.

\section{Spherical harmonics in spherical coordinates} \label{sec:spherical-coor}
\setcounter{equation}{0}

The polar coordinates $(x_1,x_2) = (r \cos \t, r\sin \t)$, $r \ge 0$, $0 \le \t \le 2 \pi$, give
coordinates for $\SS^1$ when $r =1$. The high-dimensional analogue is the spherical
polar coordinates defined by  
\begin{equation} \label{eq:sph-polar}
\begin{cases}
\quad x_{1} = r\sin \theta _{d-1}\ldots \sin \theta _{2}\sin \theta _{1},  \\
\quad x_{2} = r\sin \theta _{d-1}\ldots \sin \theta _{2}\cos \theta _{1}, \\
\quad  \cdots  \\
x_{d-1} = r\sin \theta _{d-1}\cos \theta _{d-2}, \\
\quad x_{d} =r\cos \theta _{d-1},
\end{cases}
\end{equation}
where $r\geq 0$, $0\leq \theta_1 \le  2\pi$, $0\leq \theta_i\leq \pi$ for $i =2, \ldots, d-1$.
When $r =1$ these are the coordinates for the unit sphere $\sph$, and they are in fact defined
recursively by
$$
       x = (\xi\sin \t_{d-1} , \cos \t_{d-1})  \in \sph, \qquad \xi \in \SS^{d-2}.
$$
Let $d\s = d \s_d$ be Lebesgue measure on $\SS^{d-1}$. Then it is easy to verify that
\begin{equation} \label{eq:d_sigma-recur}
       d\s_d(x) = (\sin \t_{d-1})^{d-2}  d \t_{d-1} d \s_{d-1}(\xi).
\end{equation}
Since the Lebesgue measure of $\SS^1$ is $d \t_1$, it follows by induction that
\begin{align} \label{eq:d-omega}
   d \sigma = d\sigma_d  =   \prod_{j=1}^{d-2}\left( \sin \theta_{d-j}\right)^{d-j-1}
 d\theta_{d-1}\ldots d\theta_2 d\theta _1
\end{align}
in the spherical coordinates \eqref{eq:sph-polar}. Furthermore, \eqref{eq:d_sigma-recur} shows that
\begin{equation} \label{eq:integral-d_d-1}
       \int_{\sph} f(x) d\s_d(x) = \int_0^\pi \int_{\SS^{d-2}} f(\xi\sin \t, \cos \t) d\s_{d-1}(\xi) (\sin \t)^{d-2} d\t.
\end{equation}

The orthogonality \eqref{eq:ultrasphericalOP} of the Gegenbauer polynomials can be written as
\begin{equation} \label{eq:GegenOP-trig}
    \int_0^\pi C_n^\l(\cos \t) C_m^\l (\cos\t) (\sin \t)^{d-2} d\t= \frac{\sqrt{\pi} \Gamma(\frac{d-1}{2})}{\Gamma(\frac{d}{2})}
      h_n^\l \delta_{m,n}, \quad \l = \tfrac{d-2}{2}.
\end{equation}
Together with \eqref{eq:integral-d_d-1}, this allows us to write down a basis of spherical harmonics in
terms of the Gegenbauer polynomials in the spherical coordinates.

\begin{thm} \label{thm:sph-harmonic-basis}  
For $d > 2$ and $\a \in \NN^d_0$, define
\begin{equation} \label{sph-harmonic-basis}
Y_\a (x) := [h_\a ]^{-1} r^{|\a|} g_\a (\t_1) \prod_{j=1}^{d - 2}
    (\sin \t_{d-j})^{|\a^{j +1}|} C^{\l_j}_{\a_j} (\cos \t_{d-j}),
\end{equation}
where $g_\a (\t_1) = \cos \a_{d-1} \t_1$ for $\a_d=0$, $\sin \a_{d-1} \t_1$ for $\a_d= 1$,
$|\a^j| = \a_j + \ldots +\a_{d-1}$, $\l_j = |\a^{j+1}| + (d - j - 1)/2$, and
$$
[h_\a]^2 := b_\a \prod_{j=1}^{d-2} \frac{\a_j! (\frac{d-j +1}{2})_{|\a^{j+1}|}
 (\a_j + \l_j )} {(2 \l_j )_{\a_j} ( \frac{d-j}{2})_{|\a^{j +1}|} \l_j}
$$
in which $b_\a = 2$ if $\a_{d-1} + \a_d > 0$,  while $b_\a = 1$ otherwise.
Then $\{Y_\a : |\a| = n, \a_d = 0, 1\}$ is an orthonormal basis of
$\CH_n^d$; that is, $
    \la Y_\a , Y_\beta \ra_{\sph} = \d_{\a,\b}.
$
\end{thm}

\begin{proof}
To see that $Y_\a$ is a homogeneous polynomial, we use, by \eqref{eq:sph-polar}, the relation
$
 \cos \t_k = {x_{k+1}}/{\sqrt{x_1^2+\ldots + x_{k+1}^2}}
$
for $1 \le k \le d-1$ to rewrite \eqref{sph-harmonic-basis} as
$$
Y_\a (x) = [h_\a]^{-1} g(x) \prod_{j=1}^{d-2} (x_1^2 + \ldots + x_{d-j+1}^2)^{\a_j/2} C_{\a_j}^{\l_j}
\bigg( \frac{x_{d-j+1}}{\sqrt{x_1^2 + \ldots +x_{d-j+1}^2} }\bigg)
$$
where $g(x) = \rho^{\a_{d-1}} \cos \a_{d-1} \t_1$ for $\a_d =0$, $\rho^{\a_{d-1}} \cos \a_{d-1} \t_1$
for $\a_d =1$, with $\rho = \sqrt{x_1^2+x_2^2}$. Since $x_1 = \rho \sin \t_1$ and $x_2 = \rho \cos \t_1$
by \eqref{eq:sph-polar}, $g(x)$ is either the real part or the imaginary part of $(x_2 + i x_1)^{\a_{d-1}}$,
which shows that it is a homogeneous polynomial of degree $\a_{d-1}$ in $x$. Since $C_n^\l(t)$
is even when $n$ is even and odd when $n$ is odd, we see that $Y_\a \in \P_n^d$. Using
\eqref{eq:integral-d_d-1}, we see that
\begin{align*}
   \la Y_\a, Y_{\a'} \ra_{\sph} = & \frac{h_\a^{-1} h_{\a'}^{-1}}{\o_d} \int_{0}^{2\pi} g_\a(\t_1) g_{\a'}(\t_1) d\t_1 \\
   &\times   \prod_{j=1}^{d-2} \int_{0}^\pi C_{\a_j}^{\l_j}(\cos \t_{d-j})C_{\a_j'}^{\l_j}(\cos \t_{d-j})
         (\sin \t_{d-j})^{2\l_j} d\t_{d-j}
 \end{align*}
from which the orthogonality follows from the orthogonality of the Gegenbauer polynomials
\eqref{eq:integral-d_d-1} and that of $\cos m \t$ and $\sin m \t$ on $[0, 2 \pi)$, and the
formula for $h_\a$ follows from the normalization constant of the Gegenbauer polynomial. 
\end{proof}

For $d =2$ and the polar coordinates $(x_1, x_2) = (r\cos \t, r\sin \t)$, it is easy to see that $\nabla_0 = \partial_\t$,
where $\partial_\t = \partial /\partial \t$. Hence by \eqref{eq:Delta=nabla2}, the Laplace-Beltrami operator for
$d =2$ is $\Delta_0 = \partial_\t^2$. Using \eqref{Beltrami-iterate} iteratively, we see that the Laplace-Beltrami operator
$\Delta_0$ has an explicit formula in the spherical coordinates \eqref{eq:sph-polar},
\begin{align} \label{Beltrami-spherical}
\Delta_0 = &\, \frac{1}{\sin^{d-2}\t_{d-1}} \frac{\partial}{\partial \t_{d-1}} \left[
  \sin^{d-2} \t_{d -1} \frac{\partial}{\partial \t_{d-1}} \right] \\
    & + \sum_{j =1}^{d-2} \frac{1}{\sin^2 \t_{d-1} \cdots \sin^2 \t_{j +1} \sin^{j-1} \t_j }
    \frac{\partial} {\partial \t_j} \left[ \sin^{j-1}
    \t_j  \frac{\partial} {\partial \t_j} \right].\notag
\end{align}

\section{Spherical harmonics in two and three variables} \label{sec:sph-harmonic-d=3}
\setcounter{equation}{0}

Since spherical harmonics in two and three variables are used most often
in applications, we state their properties in this section.

\subsection{Spherical harmonics in two variables}  
For $d = 2$, $\dim \CH_n^2 = 2$. An orthogonal basis of $\CH_n^2$ is
given by the real and imaginary parts of $(x_1 + i x_2)^n$, since both are
homogeneous of degree $n$ and are harmonic as the real and imaginary
parts of an analytic function. In polar coordinates $(x_1, x_2) = (r \cos \t, r \sin \t)$
of $\RR^2$, this basis is given by
\begin{equation} \label{sph-harmon-d=2}
   Y_n^{(1)} (x) = r^n \cos n \theta, \qquad Y_n^{(2)}(x) = r^n \sin n \theta.
\end{equation}
Hence, restricting to the circle $\SS^1$, the spherical harmonics are precisely
the cosine and sine functions. In particular, spherical harmonic expansions on $\SS^1$
are the classical Fourier expansions in cosine and sine functions.

As homogeneous polynomials, the basis \eqref{sph-harmon-d=2} is given
explicitly in terms of the Chebyshev polynomials $T_n$ and $U_n$ 
defined by
$$
    T_n(t) = \cos n \t \quad \hbox{and} \quad U_n(t) = \frac{\sin (n+1)\t}{\sin \t}, \quad \hbox{where}\quad t = \cos \t,
$$
which are related to the Gegenbauer polynomials:  $U_n(t) = C_n^1(t)$ and
$$
       \lim_{\l \to 0+} \frac{1}{\l} C_n^\l (x) = \frac{2}{n} T_n (x).
$$
The basis in \eqref{sph-harmon-d=2} can be rewritten then as
\begin{equation} \label{sph-harmon-d=2B}
  Y_n^{(1)}(x) = r^n T_n \left(\frac{x_1}{r}\right), \quad  Y_n^{(2)}(x) = r^{n-1} x_2 U_{n-1} \left(\frac{x_1}{r}\right),
\end{equation}
which shows explicitly that these are homogeneous polynomials since  $r = \sqrt{x_1^2+x_2^2}$ and
both $T_n(t)$ and $U_n(t)$ are even if $n$ is even, odd if $n$ is odd.

When $d = 2$, the zonal polynomial is given by $T_n(\la x,y\ra) = \cos n (\t-\phi)$ and
the addition formula \eqref{addition_formula} becomes, by \eqref{sph-harmon-d=2},
the addition formula
$$
  \cos n\t \cos n\phi + \sin n \t \sin n \phi = \cos n (\t - \phi).
$$
The expression  \eqref{Beltrami-radial} of the Laplace operator in polar coordinates becomes
$$
      \Delta = \frac{d^2}{dr}+ \frac{1}{r} \frac{d}{dr} + \frac{1}{r^2} \frac{d^2}{d \t^2};
$$
in particular, the Laplace--Beltrami operator on $\SS^1$ is simply $\Delta_0 =  d^2 / d \t^2$.

\subsection{Spherical harmonics in three variables} 
The space $\CH_n^3$ of spherical harmonics of degree $n$ has dimension $2n + 1$.
For $d =3$ the spherical polar coordinates \eqref{eq:sph-polar} are written as
\begin{equation} \label{spherical-polar-d=3}
\begin{cases}
   x_1 = r \sin \t \sin \phi, \\
   x_2 = r \sin \t \cos \phi, \\
   x_3 = r \cos \t,
\end{cases}, \quad  0 \le \t \le \pi, \quad 0 \le \phi < 2\pi, \quad r > 0.
\end{equation}
The surface area of $\SS^2$ is $4 \pi$ and the integral over $\sph$ is parameterized by
\begin{equation} \label{integral-sph-d=3}
  \int_{\SS^2} f(x) d\s = \int_0^\pi \int_0^{2 \pi} f (\sin\t \sin \phi, \sin\t \cos \phi, \cos \t) d\phi \sin\t d\t.
\end{equation}
The orthogonal basis \eqref{sph-harmonic-basis} in spherical coordinates becomes
\begin{equation}\label{sph-harmonic-basis-d=3}
\begin{split}
   Y_{k,1}^n (\t, \phi) & = (\sin \t)^k C_{n-k}^{k+ \frac12}(\cos \t) \cos k\phi, \quad 0 \le  k \le  n,\\
  Y_{k,2}^n (\t, \phi)  & = (\sin \t)^k C_{n-k}^{k+ \frac12}(\cos \t)  \sin k \phi, \quad 1 \le k \le n.
\end{split}
\end{equation}
Their $L^2 (\SS^2)$ norms can be deduced from \eqref{sph-harmonic-basis}. This basis is often written
in terms of the associated Legendre polynomials $P_n^k(t)$ defined by
$$
   P_n^k(x):= (-1)^n (1-x^2)^{k/2} \frac{d^k}{dx^k} P_n(x) = (2k-1)!! (-1)^n (1-x^2)^{k/2} C_{n-k}^{k+1/2}(x),
$$
where $P_n(t) = C_n^0(t)$ denotes the Legendre polynomial of degree $n$ (see Appendix B for
properties of $P_n$ and $P_n^k$), and in terms of $\{e^{ik\phi}, e^{-ik \phi} \}$
instead of $\{\cos k\phi, \sin k\phi \}$. In this way an orthonormal basis of $\CH_n^3$ is given by
\begin{equation} \label{ON-basis-d=3}
  Y_{k,n} (\t, \phi) = \( \frac{(2n + 1)(n- k)!}{(n + k)!} \)^{1/2}P_n^k (\cos \t) e^{ik \phi},
    \quad   - n \le  k \le n.
\end{equation}
The addition formula \eqref{addition_formula} then reads, assuming that $x$ and $y$ have spherical
coordinates $(\t, \phi)$ and $(\t' , \phi')$, respectively,
\begin{equation}\label{addition_d=3}
   \sum_{k= -n}^n  Y_{k,n} (\t, \phi)Y_{k,n} (\t', \phi') = (2n + 1)P_n (\la x, y \ra).
\end{equation}
In terms of the coordinates \eqref{spherical-polar-d=3}, the Laplace--Beltrami operator is given by
 \begin{equation}\label{Beltrami_d=3}
\Delta_0 = \frac{1}{\sin \t} \frac{\partial}{\partial \t} \( \sin \t  \frac{\partial}{\partial \t} \) +
    \frac{1}{\sin^2 \t} \frac{\partial^2}{\partial \phi^2}
\end{equation}
as seen from \eqref{Beltrami-spherical}.

\section{Representation of the rotation group} \label{sec:gp-reps}
\setcounter{equation}{0}

In this section we show that the representation of the group $SO(d)$ in spaces
of harmonic polynomials is irreducible.  

A representation of $SO(d)$ is a homomorphism from $SO(d)$ to the
group of nonsingular continuous linear transformations of $L^2
(\sph)$. We associate with each element $Q \in SO(d)$ an operator
$T (Q)$ in the space of $L^2(\sph)$, defined by   
\begin{equation}\label{left_rep}
  T (Q)f (x) = f (Q^{-1} x), \qquad x \in \sph.
\end{equation}
Evidently, for each $Q \in SO(d)$, $T(Q)$ is a nonsingular linear transformation of
$L^2 (\sph)$ and $T$ is a homomorphism,
$$
T (Q_1 Q_2 ) = T (Q_1 )T (Q_2 ), \quad \forall Q_1 , Q_2 \in SO(d).
$$
Thus, $T $ is a representation of $SO(d)$. Since $d\s$ is invariant
under rotations, $\|T(Q)f \|_2 = \|f\|_2$ in the $L^2(\sph)$ norm,
so that $T (Q)$ is unitary.

A linear space $\U$ is called invariant under $T$ if $T (Q)$
maps $\U$ to itself for all $Q \in SO(d)$. The null space and $L^2(\sph)$ itself are
trivial invariant subspaces. A representation $T$ is called irreducible if it has only trivial
invariant subspaces. The space $\CH_n^d$ is an invariant subspace of $T$ in
\eqref{left_rep}.

Let $T_{n,d}$ denote the representation of $SO(d)$ corresponding to $T$
in the invariant subspace $\CH_n^d$. We want to show that $T_{n,d}$ is irreducible.

\begin{lem}\label{lem:harmonic_xd}
A spherical harmonic $Y \in \CH_n^d$ is invariant under all rotations in
$SO(d)$ that leaves $x_d$ fixed if and only if
\begin{equation} \label{eq:harmonic_xd}
 Y (x) = c \|x\|^n C_n^\l \Big (\frac{x_d}{\|x\|}\Big), \qquad \l = \frac{d -2}{2}.
\end{equation}
where $c$ is a constant.
\end{lem}

\begin{proof}
If $Y$ is invariant under rotations that fix $x_d$ and is a homogeneous
polynomial of degree $n$, then it can be written as
$$
Y (x) =  \sum_{0 \le j\le n/2} b_j x_d^{n - 2j} (x_1^2 + \ldots + x_{d-1}^2)^j =
\sum_{0 \le j \le n/2} c_j x_d^{n-2j} \|x\|^{2j}.
$$
where the second equal sign follows from expanding $(\|x\|^2-
x_d^2)^j$ and changing the order of summation. Since $Y$ is
harmonic, computing $\Delta Y (x) = 0$ shows that $c_j$ satisfies
the recurrence relation
$$
   4(j + 1)(n - j - 1)c_{j +1} + (n - 2j )(n - 2j - 1)c_j = 0.
$$
Solving the recurrence equation for $c_j$ we conclude that
$$
Y (x) = c_0 \sum_{0 \le j \le n/2} \frac{ (- \frac{n}{2})_j ( \frac{1-n}{2} )_j}
    {j !(1- n-d -2)_j} x_d^{n-2j} \|x\|^{2j}.
$$
Consequently, \eqref{eq:harmonic_xd} follows from the formula
\eqref{Gegen-2F1} for Gegenbauer polynomials. Since the function
$Y$ in \eqref{eq:harmonic_xd} is clearly invariant under all
rotations that fix $x_d$ and we have just shown that it is
harmonic, the proof is complete.
\end{proof}

\begin{thm} \label{thm:irreducible}
The representation $T_{n,d}$ of $SO(d)$ on $\CH_n^d$ is irreducible.
\end{thm}

\begin{proof}
Assume that $\U$ is an invariant subspace of $\CH_n^d$ and $\U$ is
not a null space. Let $\{Y_j : 1 \le j \le M \}$, $M \le \dim
\CH_n^d$, be an orthonormal basis of $\U$. Following the proof of
Lemma \ref{lem:zonal}, there is a polynomial $F(t)$ of one
variable such that $\sum_{j =1}^M Y_j (x)Y_j (y) = F ( \la x, y
\ra)$. In particular, setting $y = e_d = (0, . . . , 0, 1)$ shows
that $F ( \la x, e_d \ra)$ is in $\CH_n^d$ and it is evidently
invariant under rotations in $SO(d)$ that fix $x_d$. Hence, by
Lemma \ref{lem:harmonic_xd} $F (\la x, e_d \ra) = c \|x\|^n C_n^\l
(\frac{x_d}{\|x\|})$. In particular, this shows that $\|x\|^n
C_n^\l (\frac{x_d}{\|x\|}) \in \U$. On the other hand, let
$\U^\perp$ denote the orthogonal complement of $\U$ in $\CH_n^d$.
If $f \in \U^\perp$ and $g \in \U$, then $\la T(Q)f , g \ra_{\sph}
= \la f , T(Q^{-1})g \ra_{\sph}  = 0$, which shows that $\U^\perp$
is also an invariant subspace of $\CH_n^d$. Applying the same
argument as for $\U$ shows then $\|x\|^n C_n^\l
(\frac{x_d}{\|x\|}) \in \U^\perp$, which contradicts to $\U \cap \U^\perp = \{0\}$.
Thus, $\U$ must be
trivial.
 \end{proof}

\section{Angular derivatives and the Laplace--Beltrami operator} \label{sec:Beltrami-Dij}
\setcounter{equation}{0}

Consider the case $d = 2$ and the polar coordinates $(x_1, x_2) = (r\cos \t, r\sin \t)$. Let $\partial_r = \partial/\partial r$
and $\partial_\t = \partial /\partial \t$, while we retain $\partial_1$ and $\partial_2$ for the partial derivatives with respect
to $x_1$ and $x_2$. It then follows that
$$
  \partial_1  = \cos \t \partial_r - \frac{\sin \t}{r} \partial_\t, \quad  \partial_2  = \sin \t \partial_r +  \frac{\cos \t}{r} \partial_\t.
$$
From these relations it follows easily that the angular derivative $\partial_\t$ can also be written as
$\partial_\t  = x_1 \partial_2 - x_2 \partial_1$ and the operator $\Delta_0$ is $\Delta_0 = \partial_\t^2$.
We introduce angular derivatives in higher dimensions as follows. 

\begin{defn} \label{def:Dij} 
For $x \in \RR^d$ and $1 \le i \ne j \le d$, define
\begin{equation} \label{defn:Dij}
        D_{i,j} : = x_i \partial_j - x_j \partial_i = \frac{\partial}{\partial \t_{i,j}},
\end{equation}
where $\t_{i,j}$ is the angle of polar coordinates in the $(x_i,x_j)$-plane, defined by $(x_i,x_j) = r_{i,j}(\cos \t_{i,j}, \sin \t_{i,j})$,
$r_{i,j} \ge 0$ and $0 \le \t_{i,j} \le 2 \pi$.
\end{defn}

By its definition with partial derivatives on $\RR^d$,  $D_{i,j}$ acts on $\RR^d$, yet the second
equality in \eqref{defn:Dij} shows that it acts on the sphere $\sph$. Thus, for $f$ defined on $\RR^d$,
\begin{equation}\label{Dij-RdSd}
    (D_{i,j} f)(\xi) = D_{i,j} [ f(\xi)], \qquad  \xi  \in \sph,
\end{equation}
where the right-hand side means that $D_{i,j}$ is acting on $f(\xi)$.

Since $D_{j,i} = - D_{i,j}$, the number of distinct operators $D_{i,j}$ is $\binom{d}{2}$.
The operator $\Delta_0$ can be decomposed in terms of them.

\begin{thm} \label{lem: Laplace-Beltrami2} 
On $\sph$, $\Delta_0$ satisfies the decomposition
\begin{equation}\label{Beltrami2}
           \Delta_0 = \sum_{1 \le i<j \le d} D_{i,j}^2.
\end{equation}
\end{thm}

\begin{proof}
Let $F (x) = f \( \frac{x}{\|x\|}\)$. A straightforward computation shows that
\begin{align*}
 \sum_{1\le i<j \le d} D_{i,j}^2 F(x) =   (\Delta F )(x) - \frac{1}{\|x\|^2}
 \sum_{i=1}^d \sum_{j=1}^d x_i x_j \frac{ \partial^2 F}{\partial x_i \partial x_j}
   - \frac{d - 1}{\|x\|} \sum_{i=1}^d \frac{ \partial F} {\partial x_i}.
\end{align*}
Consequently, restricting to $\sph$ and comparing with \eqref{Beltrami-explicit}, \eqref{Beltrami2}
follows.
\end{proof}

Let $Q_{i,j,\t}$ denote a rotation by the angle $\t$ in the $(x_i , x_j)$-plane, oriented so that $(x_i, x_j)
= (s \cos \t, s \sin \t)$. 
Then $T (Q_{i,j,\t})$, defined in \eqref{left_rep}, maps $f$ into $T (Q_{i,j,\t})f (x)
= f (Q_{i,j,-\t} x)$. Written explicitly, for example, for $(i,j) = (1,2)$, we have 
\begin{equation} \label{eq:TQij}
T (Q_{1,2,\t} f )(x) = f (x_1 \cos \t + x_2 \sin \t,
    - x_1 \sin \t + x_2 \cos \t, x_3 , \ldots, x_d).
\end{equation}
Then $D_{i,j}$ is the infinitesimal operator of $T(Q)$,
\begin{equation} \label{partial_ij}
\frac{dT (Q_{i,j,\t})}{d \t} \Big \vert_{\t =0} =
 x_i \frac{\partial}{\partial x_j}   - x_j \frac{\partial}{\partial x_i}  = D_{i,j},
\end{equation}
where the first equality follows from \eqref{eq:TQij}.
The infinitesimal operator plays an important role in representation theory; see,
for example, \cite{Vile}.

The operators $D_{i,j}$ will play an important role for approximation theory on the sphere. We
state  several more properties of these operators.

\begin{lem} \label{lem:Dij-Hnd}
For $1 \le i < j \le d$, the operators $D_{i,j}$ commute with $\Delta_0$. In
particular, $D_{i,j}$ maps $\CH_n^d$ to itself.
\end{lem}

\begin{proof}
By symmetry, we only to show only that $D_{1,2}$. Let $[A,B]= A B - BA$ denote the commutator
of $A$ and $B$. A quick computation shows that
$$
   [\partial_j, D_{k,l}] = \delta_{j,k} \partial_l- \delta_{j,l}\partial_k, \quad
   [x_i, D_{k,l}] =  \delta_{i,k} x_l - \delta_{i,l} x_k,
$$
from which it is easy to see that, for  $1 \le i < j \le d$, $1 \le k < l \le d$, we have 
\begin{equation} \label{eq:Dijcommuter}
   [D_{i,j}, D_{k,l}]  = - \delta_{i,k} D_{j, l}+\delta_{i,l} D_{j, k} +\delta_{j,k} D_{i, l} - \delta_{j,l} D_{i, k}.
\end{equation}
Using \eqref{eq:Dijcommuter}, a simple computation shows that $[D_{1,2}, D_{1,l}^2]  =
- (D_{1,l}D_{2,l}+ D_{2,l}D_{1,l})$ and $[D_{1,2}, D_{2,l}^2] = D_{1,l}D_{2,l}+ D_{2,l}D_{1,l}$,
so that $[D_{1,2}, D_{1,l}^2 + D_{2,l}^2] =0$ for $l \ge 2$. Moreover, by \eqref{eq:Dijcommuter},
$[D_{1,2}, D_{k,l}^2] =0$ whenever $3\le k < l \le d$. Summing over $(k,l)$ for
$1 \le k < l \le d$ then proves $[D_{1,2},\Delta_0]=0$.
\end{proof}

\begin{prop} \label{lem:parts-Dij}
For $f, g \in C^1(\sph)$ and $1\leq i\neq j\leq d$,
\begin{equation} \label{eq:parts-Dij}
   \int_{\sph} f(x) D_{i,j} g(x) d\s(x) =  -  \int_{\sph} D_{i,j} f (x)g(x)d\s(x).
\end{equation}
\end{prop}

\begin{proof}
By the rotation invariance of the Lebesgue measure $d\s$, we obtain, for
every $\t\in [-\pi,\pi]$,
$$
  \int_{\sph} f(x) g (Q_{i,j,-\t} x)\, d\s(x)=\int_{\sph}f(Q_{i,j,\t} x) g(x)\, d\s(x).
$$
Differentiating  both sides of this identity with respect to $\t$ and evaluating the
resulting equation at $\t=0$ leads to, by \eqref{partial_ij}, the desired equation \eqref{eq:parts-Dij}.
\end{proof}

The equation \eqref{eq:parts-Dij} allows us to define distributional derivatives
$D_{i,j}^r$ on $\sph$ for $r\in\NN$ via the identity, wtih $g\in C^\infty(\sph)$,
\begin{equation} \label{eq:distributiveDij}
   \int_{\sph} D_{i,j}^r f(x) g(x) d\s(x) =  (-1)^r  \int_{\sph}  f(x)D_{i,j}^r
   g(x)d\s(x).
\end{equation}

Summing \eqref{eq:distributiveDij} with $r =2$ over $i < j$ and applying \eqref{Beltrami2}, it follows immediately
that $\Delta_0$ is self-adjoint, which can also be deduced from Theorem \ref{thm:eigen-eqn-sph}.

\begin{cor}
For $f, g  \in C^2(\sph)$,
$$
    \int_{\sph} f(x) \Delta_0 g(x) d\s =  \int_{\sph} \Delta_0 f(x) g(x) d\s.
$$
\end{cor}

The spherical gradient $\nabla_0$ is a vector of first order differential operator on the sphere, which
can be written in terms of $D_{i,j}$ as follows.  

\begin{lem}
For $f\in C^1(\sph)$ and $1 \le j \le d$, the $j$-th component of $\nabla_0f$ satisfies
\begin{equation}\label{nabla0-Dij}
       (\nabla_0)_j f(\xi) = \sum_{1 \le i \le d, i \ne j} \xi_i D_{i,j} f(\xi), \quad \xi \in \sph.
\end{equation}
Furthermore, for $f,g \in C^1(\sph)$, the following identity holds:
\begin{equation}\label{nabla0-Dij-sum}
  \nabla_0 f(\xi) \cdot \nabla_0 g(\xi)  = \sum_{1 \le i < j \le d} D_{i,j} f(\xi) D_{i,j} g(\xi), \quad \xi \in \sph.
\end{equation}
\end{lem}

\begin{proof}
Let $F(y) = f(y/\|y\|)$. From the definition and $\|\xi\| =1$, we obtain
\begin{align} \label{nabla0-Dij2}
 (\nabla_0)_j f(\xi) & = \frac{\partial}{\partial x_j} F (\xi) = \partial_j f - \xi_j \sum_{i=1}^d \xi_i \partial_i f  \\
     & = \partial_j f \sum_{i=1}^d \xi_i^2 - \xi_j \sum_{i=1}^d \xi_i \partial_i f = \sum_{i=1}^d \xi_i D_{i,j} f, \notag
\end{align}
which gives \eqref{nabla0-Dij} since $D_{i,i} f = 0$. \eqref{nabla0-Dij2} means that
\begin{equation} \label{eq:nbala0-nabla}
   \nabla_0 f(\xi) =\nabla f(\xi)-\xi ( \xi\cdot \nabla f).
\end{equation}

Since $\xi\cdot \nabla_0 f(\xi) =0$, using \eqref{nabla0-Dij2} and the definition of
$D_{i,j}$, it follows  that
\begin{align*}
    \nabla_0 f \cdot \nabla_0 g= \nabla_0 f \cdot \nabla g  =  \sum_{j=1}^d \partial_j f \partial_j g - \sum_{j =1}^d \sum_{i =1}^d \xi_i \xi_j \partial_i f \partial_j g
        =\sum_{i<j } D_{i,j} f(\xi) D_{i,j} g(\xi),
\end{align*}
where the second equality uses $\|\xi\| =1$.
\end{proof}

As an application, we state an integration by parts formula on the sphere.

\begin{prop} \label{prop:parts-sphere}
For $f, g \in C^1(\sph)$,
\begin{equation} \label{eq:parts-sphere}
    \int_{\sph} f(x) \nabla_0 g(x) d\s = - \int_{\sph} \left(\nabla_0 f(x) - (d-1)x f(x) \right) g(x) d\s.
\end{equation}
Furthermore, for $f \in C^2(\sph)$ and $g \in C^1(\sph)$,
\begin{equation} \label{eq:parts-Delta0-nabla0}
   \int_{\sph} \nabla_0 f \cdot \nabla_0 g d\s = -  \int_{\sph} \Delta_0 f(x) g(x) d\s.
\end{equation}
\end{prop}

\begin{proof}
Using the expression \eqref{nabla0-Dij2} and applying \eqref{eq:parts-Dij}, we obtain
\begin{align*}
 \int_{\sph} f(x) (\nabla_0)_j g(x) d\s & = \sum_{1 \le i \le d, i \ne j} \int_{\sph} x_i f(x) D_{i,j} g(x) d\s \\
    & = - \sum_{1 \le i \le d, i \ne j} \int_{\sph} D_{i,j} ( x_i f(x)) g(x) d\s.
\end{align*}
By the chain rule, recalling \eqref{Dij-RdSd} if necessary, we have $D_{i,j}  ( x_i f(x)) = x_i D_{i,j} f - x_j f(x)$.
Hence we obtain
\begin{align*}
 \int_{\sph} f(x) (\nabla_0)_j g(x) d\s & = -  \int_{\sph} \biggl( \sum_{1 \le i \le d, i \ne j}x_i D_{i,j} f(x)  - (d-1) f(x) \biggr) g(x) d\s,
\end{align*}
which gives the $j$-th component of \eqref{eq:parts-sphere}. Equation \eqref{eq:parts-Delta0-nabla0}
follows  immediately from \eqref{nabla0-Dij-sum}, \eqref{eq:parts-Dij} and \eqref{Beltrami2}.
\end{proof}

\section{Notes}
\setcounter{equation}{0}

Spherical harmonics appear in many disciplines and in many different branches of mathematics.
Many books contain parts of the theory of spherical harmonics. Our treatment covers 
what is needed for harmonic analysis and approximation theory in this book. Below we
comment on some books that we have consulted.

A classical treatise on spherical harmonics is \cite{Hob}, a good source for classical results.
A short but nice expository is \cite{Muller66}, which was later expanded into \cite{Muller}.
The reference book \cite{HTF} contains a chapter on spherical harmonics. A rich resource for
spherical harmonics in Fourier analysis is \cite{SW}. Applications to and connections with
group representation are studied extensively in \cite{Vile}; see also \cite{Helg}. For their
role in the context of orthogonal polynomials of several variables, see \cite{DX} as well as
\cite{HTF}. The book \cite{AAR} contains a chapter on spherical harmonics in light of special
functions. The theory of harmonic functions is treated in \cite{ABR}, including material on
spherical harmonics. The book \cite{Gro} deals with spherical harmonics in geometric applications.
Finally, the recent book \cite{AtH} provides an introduction to spherical harmonics and
approximation on the sphere from the perspective of applications in numerical analysis.

Aside from their role in representation theory, the operators $D_{i,j}$ do not seem to have received much
attention in analysis. Most of the materials in Section \ref{sec:Beltrami-Dij} have not previously
appeared in books. These operators play an important role in our development of approximation
theory on the sphere.

\bibliographystyle{amsalpha}

\newpage

\noindent
\begin{center}
{\huge Approximation Theory and Harmonic Analysis\\ on Spheres and Balls} 
\end{center}
\bigskip

\centerline {\Large F. Dai and Y. Xu}
\medskip

\centerline {\Large Springer Monographs in Mathematics,  Springer,  New York 2013.}

\bigskip

\bigskip

\centerline{\Large Table of Contents}
 
\contentsline {chapter}{Preamble}{xvii}
\contentsline {chapter}{\numberline {1}Spherical Harmonics}{1}
\contentsline {section}{\numberline {1.1}Space of spherical harmonics and orthogonal bases}{1}
\contentsline {section}{\numberline {1.2}Projection operators and Zonal harmonics}{7}
\contentsline {section}{\numberline {1.3}Zonal basis of spherical harmonics}{11}
\contentsline {section}{\numberline {1.4}Laplace-Beltrami operator}{14}
\contentsline {section}{\numberline {1.5}Spherical harmonics in spherical coordinates}{17}
\contentsline {section}{\numberline {1.6}Spherical harmonics in two and three variables}{19}
\contentsline {subsection}{\numberline {1.6.1}Spherical harmonics in two variables}{19}
\contentsline {subsection}{\numberline {1.6.2}Spherical harmonics in three variables}{20}
\contentsline {section}{\numberline {1.7}Representation of the rotation group}{22}
\contentsline {section}{\numberline {1.8}Angular derivatives and Laplace-Beltrami operator}{23}
\contentsline {section}{\numberline {1.9}Notes}{27}
\contentsline {chapter}{\numberline {2}Convolution and Spherical Harmonic Expansion}{29}
\contentsline {section}{\numberline {2.1}Convolution and translation operators}{29}
\contentsline {section}{\numberline {2.2}Fourier orthogonal expansions}{33}
\contentsline {section}{\numberline {2.3}The Hardy-Littlewood Maximal function}{37}
\contentsline {section}{\numberline {2.4}Spherical harmonic series and Ces\`aro means}{40}
\contentsline {section}{\numberline {2.5}Convergence of Ces\`aro means: further results}{43}
\contentsline {section}{\numberline {2.6}highly localized kernels}{44}
\contentsline {section}{\numberline {2.7}Notes and further results}{50}
\contentsline {chapter}{\numberline {3}Littlewood-Paley Theory and Multiplier Theorem}{53}
\contentsline {section}{\numberline {3.1}Analysis on homogeneous spaces}{53}
\contentsline {section}{\numberline {3.2}The Littlewood-Paley theory on the sphere}{56}
\contentsline {section}{\numberline {3.3}The Marcinkiewitcz multiplier theorem}{64}
\contentsline {section}{\numberline {3.4}The Littlewood-Paley inequality}{67}
\contentsline {section}{\numberline {3.5}The Riesz transform on the sphere}{71}
\contentsline {subsection}{\numberline {3.5.1}Fractional Laplace-Beltrami operator and Riesz transform}{71}
\contentsline {subsection}{\numberline {3.5.2}Proof of Lemma 3.5.2\hbox {}}{74}
\contentsline {section}{\numberline {3.6}Notes and further results}{78}
\contentsline {chapter}{\numberline {4}Approximation on the Sphere}{79}
\contentsline {section}{\numberline {4.1}Approximation by trigonometric polynomials}{80}
\contentsline {section}{\numberline {4.2}Modulus of smoothness on the unit sphere}{85}
\contentsline {section}{\numberline {4.3}A key lemma}{89}
\contentsline {section}{\numberline {4.4}Characterization of best approximation}{93}
\contentsline {section}{\numberline {4.5}$K$-functionals and Approximation in Sobolev space}{95}
\contentsline {section}{\numberline {4.6}Computational examples}{97}
\contentsline {section}{\numberline {4.7}Other moduli of smoothness}{99}
\contentsline {section}{\numberline {4.8}Notes and further results}{101}
\contentsline {chapter}{\numberline {5}Weighted Polynomial Inequalities}{105}
\contentsline {section}{\numberline {5.1}Doubling weights on the sphere}{105}
\contentsline {section}{\numberline {5.2}A maximal function for spherical polynomials}{111}
\contentsline {section}{\numberline {5.3}Marcinkiewicz-Zygmund inequalities}{114}
\contentsline {section}{\numberline {5.4}Further inequalities between sums and integrals}{118}
\contentsline {section}{\numberline {5.5}Nikolskii and Bernstein inequalities}{124}
\contentsline {section}{\numberline {5.6}Notes and further results}{126}
\contentsline {chapter}{\numberline {6}Cubature Formulas on Spheres}{127}
\contentsline {section}{\numberline {6.1}Cubature formulas}{127}
\contentsline {section}{\numberline {6.2}Product type cubature formulas on the sphere}{132}
\contentsline {section}{\numberline {6.3}Positive Cubature formulas}{135}
\contentsline {section}{\numberline {6.4}Area-regular partitions of $\mathbb {S}^{d-1}$}{140}
\contentsline {section}{\numberline {6.5} Spherical designs}{144}
\contentsline {section}{\numberline {6.6}Notes and further results}{151}
\contentsline {chapter}{\numberline {7}Harmonic Analysis Associated to Reflection Groups}{155}
\contentsline {section}{\numberline {7.1}Dunkl operators and $h$-spherical harmonics}{156}
\contentsline {section}{\numberline {7.2}Projection operator and intertwining operator}{160}
\contentsline {section}{\numberline {7.3}$h$-harmonics for a general finite reflection group}{166}
\contentsline {section}{\numberline {7.4}Convolution and $h$-harmonic series}{168}
\contentsline {section}{\numberline {7.5}Maximal function and multiplier theorem}{172}
\contentsline {section}{\numberline {7.6}Maximal function for ${\mathbb Z}_2^d$ invariant weight}{179}
\contentsline {section}{\numberline {7.7}Notes and further results}{186}
\contentsline {chapter}{\numberline {8}Boundedness of Projection Operator and Ces\`aro Means}{189}
\contentsline {section}{\numberline {8.1}Boundedness of Ces\`aro means above the critical index}{189}
\contentsline {section}{\numberline {8.2}A multiple Beta integral of the Jacobi polynomials}{191}
\contentsline {section}{\numberline {8.3}Pointwise estimate of the kernel functions}{198}
\contentsline {section}{\numberline {8.4}Proof of the main results}{200}
\contentsline {section}{\numberline {8.5}Lower bound for generalized Gegenbauer expansion}{204}
\contentsline {section}{\numberline {8.6}Notes and further results}{211}
\contentsline {chapter}{\numberline {9}Projection Operators and Ces\`aro Means in $L^p$ Spaces}{213}
\contentsline {section}{\numberline {9.1}Boundedness of projection operators}{213}
\contentsline {section}{\numberline {9.2}Boundedness of Ces\`aro means in $L^p$ spaces}{217}
\contentsline {subsection}{\numberline {9.2.1}Proof of Theorem 9.2.1\hbox {}}{218}
\contentsline {subsection}{\numberline {9.2.2}Proof of Theorem 9.2.2\hbox {}}{222}
\contentsline {section}{\numberline {9.3}Local estimates of the projection operators}{225}
\contentsline {subsection}{\numberline {9.3.1}Proof of Theorem 9.1.2\hbox {}, case I: ${\gamma }< {\sigma }_{\kappa }-\frac {d-2}{2}$}{226}
\contentsline {subsection}{\numberline {9.3.2}Proof of Theorem 9.1.2\hbox {}, case 2: ${\gamma }= {\sigma }_{\kappa }-\frac {d-2}{2}$}{234}
\contentsline {section}{\numberline {9.4}Notes and further results}{239}
\contentsline {chapter}{\numberline {10}Weighted Best Approximation by Polynomials}{241}
\contentsline {section}{\numberline {10.1} Moduli of smoothness and best approximation}{241}
\contentsline {section}{\numberline {10.2} Fractional powers of spherical $h$-Laplacian }{245}
\contentsline {section}{\numberline {10.3} $K$-functionals and best approximatioin}{248}
\contentsline {section}{\numberline {10.4} Equivalence of the first modulus with $K$-functional}{250}
\contentsline {section}{\numberline {10.5} Equivalence of the second modulus with $K$-functional}{254}
\contentsline {section}{\numberline {10.6} Further properties of moduli of smoothness}{261}
\contentsline {section}{\numberline {10.7} Notes and further results}{262}
\contentsline {chapter}{\numberline {11}Harmonic Analysis on the Unit Ball}{265}
\contentsline {section}{\numberline {11.1} Orthogonal structure on the unit ball}{265}
\contentsline {section}{\numberline {11.2} Convolution and orthogonal expansions}{272}
\contentsline {section}{\numberline {11.3} Maximal functions and multiplier theorem}{276}
\contentsline {section}{\numberline {11.4} Projection operators and Ce\`saro means on the ball}{280}
\contentsline {section}{\numberline {11.5} Near best approximation operators and highly localized kernels}{283}
\contentsline {section}{\numberline {11.6} Cubature formulas on the unit ball}{286}
\contentsline {subsection}{\numberline {11.6.1}Cubature formulas on the ball and on the sphere}{287}
\contentsline {subsection}{\numberline {11.6.2}Positive cubature formulas and MZ inequality}{288}
\contentsline {subsection}{\numberline {11.6.3}Product type cubature formulas}{290}
\contentsline {section}{\numberline {11.7} Orthogonal structure on spheres and on balls}{292}
\contentsline {section}{\numberline {11.8} Notes and further results}{294}
\contentsline {chapter}{\numberline {12}Polynomial Approximation on the Unit Ball}{297}
\contentsline {section}{\numberline {12.1} Algebraic polynomial approximation on the interval}{298}
\contentsline {section}{\numberline {12.2} First modulus of smoothness and $K$-functional}{303}
\contentsline {subsection}{\numberline {12.2.1}Projection from sphere to ball}{303}
\contentsline {subsection}{\numberline {12.2.2}Modulus of smoothness}{305}
\contentsline {subsection}{\numberline {12.2.3}Weighted $K$-functional and the equivalence}{307}
\contentsline {subsection}{\numberline {12.2.4}Main theorems}{309}
\contentsline {subsection}{\numberline {12.2.5}The moduli of smoothness on $[-1,1]$}{311}
\contentsline {subsection}{\numberline {12.2.6}Computational examples}{312}
\contentsline {section}{\numberline {12.3} The second modulus of smoothness and $K$-functional}{313}
\contentsline {subsection}{\numberline {12.3.1}Analogue of the Ditzian-Totik $K$-functional}{314}
\contentsline {subsection}{\numberline {12.3.2}Direct and inverse theorems by $K$-functional}{319}
\contentsline {subsection}{\numberline {12.3.3}Analogue of Ditzian-Totik modulus of smoothness on ${\mathbb B}^d$}{320}
\contentsline {subsection}{\numberline {12.3.4}Equivalence of $ {\omega }_\varphi ^r(f,t)_p$ and $ K_{r,\varphi }(f,t)_p$}{323}
\contentsline {section}{\numberline {12.4} Third modulus of smoothness and $K$-functional}{326}
\contentsline {section}{\numberline {12.5} Comparisons of three moduli of smoothness}{328}
\contentsline {section}{\numberline {12.6} Notes and further results}{329}
\contentsline {chapter}{\numberline {13}Harmonic Analysis on the Simplex}{333}
\contentsline {section}{\numberline {13.1} Orthogonal structure on the simplex}{333}
\contentsline {section}{\numberline {13.2} Convolution and orthogonal expansions}{337}
\contentsline {section}{\numberline {13.3} Maximal functions and multiplier theorem}{340}
\contentsline {section}{\numberline {13.4} Projection operator and Ces\`aro means}{344}
\contentsline {section}{\numberline {13.5} Near best approximation operators and highly localized kernels}{351}
\contentsline {section}{\numberline {13.6} Weighted best approximation on the simplex}{354}
\contentsline {section}{\numberline {13.7} Cubature formulas on the simplex}{356}
\contentsline {subsection}{\numberline {13.7.1}Cubature formulas on the simplex and on the ball}{357}
\contentsline {subsection}{\numberline {13.7.2}Positive cubature formulas and MZ inequality}{357}
\contentsline {subsection}{\numberline {13.7.3}Product type cubature formulas}{359}
\contentsline {section}{\numberline {13.8} Notes and further results}{360}
\contentsline {chapter}{\numberline {14}Applications}{361}
\contentsline {section}{\numberline {14.1} Highly localized tight polynomial frames on the sphere}{362}
\contentsline {section}{\numberline {14.2} Nodes distribution of positive cubature formulas}{369}
\contentsline {section}{\numberline {14.3} Positive definite functions on the sphere}{373}
\contentsline {section}{\numberline {14.4} Asymptotics for minimal discrete energy on the sphere}{381}
\contentsline {section}{\numberline {14.5} Computerized tomography}{388}
\contentsline {section}{\numberline {14.6} Notes and further results}{396}
\contentsline {chapter}{\numberline {A}Distance, Difference and Integral Formulas}{401}
\contentsline {section}{\numberline {A.1}Distance on spheres, balls and simplexes}{401}
\contentsline {section}{\numberline {A.2}Euler angles and rotations}{403}
\contentsline {section}{\numberline {A.3}Basic properties of difference operators }{404}
\contentsline {section}{\numberline {A.4}Ces\`aro means and Difference operators}{406}
\contentsline {section}{\numberline {A.5}Integrals over spheres and balls}{409}
\contentsline {chapter}{\numberline {B}Jacobi and Related Orthogonal Polynomials}{413}
\contentsline {section}{\numberline {B.1}Jacobi polynomials}{413}
\contentsline {section}{\numberline {B.2}Gegenbauer polynomials}{416}
\contentsline {section}{\numberline {B.3}Generalized Gegenbauer polynomials}{418}
\contentsline {section}{\numberline {B.4}Associated Legendre polynomials}{419}
\contentsline {section}{\numberline {B.5}Estimates of normalized Jacobi polynomials}{419}
\contentsline {chapter}{References}{425}
\contentsline {chapter}{Index}{433}
\contentsline {chapter}{Symbol Index}{437}

\vskip 1in

\end{document}